\documentclass[11pt,a4paper, reqno]{article}
\usepackage[utf8]{inputenc}
\usepackage[T1]{fontenc}
\usepackage{amsfonts}
\usepackage{amssymb}
\usepackage{enumerate}
\usepackage{graphicx}
\usepackage{mathtools}
\usepackage{amsmath, amsthm}
\usepackage{tikz}
\usepackage{color}
\usepackage[affil-it]{authblk}
\usepackage[sort,numbers]{natbib}
\usepackage{bbm}
\usepackage[a4paper, total={6in, 9in}]{geometry}
\usepackage{amsmath}
\usepackage{amsfonts}
\usepackage{amssymb}
\usepackage{bbm}
\usepackage{mathrsfs}
\usepackage[utf8]{inputenc}
\usepackage[english]{babel}
\usepackage{hyperref}
\usepackage{relsize}
\usepackage{accents}
\usepackage{centernot}
\usepackage[color = white]{todonotes}
\usepackage{subcaption}
\usepackage{url}

\usetikzlibrary{arrows,calc}

\newcommand\blfootnote[1]{%
  \begingroup
  \renewcommand\thefootnote{}\footnote{#1}%
  \addtocounter{footnote}{-1}%
  \endgroup
}

\usepackage{environ}
\NewEnviron{eq}{%
\begin{equation}\begin{split}
  \BODY
\end{split}\end{equation}
}

\makeatletter
  \@addtoreset{chapter}{part}
  \@addtoreset{@ppsaveapp}{part}
\makeatother

\long\def\/*#1*/{}

\usepackage{palatino}

\makeatletter
\newsavebox\myboxA
\newsavebox\myboxB
\newlength\mylenA

\newcommand*\xoverline[2][0.75]{%
    \sbox{\myboxA}{$\m@th#2$}%
    \setbox\myboxB\null
    \ht\myboxB=\ht\myboxA%
    \dp\myboxB=\dp\myboxA%
    \wd\myboxB=#1\wd\myboxA
    \sbox\myboxB{$\m@th\overline{\copy\myboxB}$}
    \setlength\mylenA{\the\wd\myboxA}
    \addtolength\mylenA{-\the\wd\myboxB}%
    \ifdim\wd\myboxB<\wd\myboxA%
       \rlap{\hskip 0.5\mylenA\usebox\myboxB}{\usebox\myboxA}%
    \else
        \hskip -0.5\mylenA\rlap{\usebox\myboxA}{\hskip 0.5\mylenA\usebox\myboxB}%
    \fi}
\makeatother

\numberwithin{equation}{section}

\usepackage{hyperref}
\hypersetup{
  colorlinks,
  linkcolor=blue,
  citecolor=red,
  filecolor=black,
  urlcolor=black,
  pdftitle={},
  pdfauthor={S. Dhara, R. van der Hofstad, J. van Leeuwaarden, S. Sen},
  pdfcreator={S. Dhara},
  pdfsubject={},
  pdfkeywords={}
}
\numberwithin{equation}{section}
\usepackage{cleveref}

\newcommand\textlcsc[1]{\textsc{\MakeLowercase{#1}}}

\def\sss{\scriptscriptstyle}

\newcommand{\floor}[1]{\ensuremath{\left\lfloor #1 \right\rfloor}}
\newcommand{\ind}[1]{\ensuremath{\mathbbm{1}{\left\{#1\right\}}}}

\newcommand{\red}[1]{{\color{red}#1}}

\newcommand{\PR}{\ensuremath{\mathbbm{P}}}
\newcommand{\E}{\ensuremath{\mathbbm{E}}}
\newcommand{\R}{\ensuremath{\mathbbm{R}}}
\newcommand{\N}{\ensuremath{\mathbbm{N}}}
\newcommand{\e}{\ensuremath{\mathrm{e}}}

\newcommand{\blue}[1]{{\color{blue}#1}}
\newcommand{\dif}{\mathrm{d}}
\newcommand{\bld}[1]{\boldsymbol{#1}}

\newcommand{\cE}{\mathcal{E}}
\newcommand{\cG}{\mathcal{G}}

\newcommand{\cX}{\mathcal{X}}

\newcommand{\T}{\mathcal{T}}

\newcommand{\bt}{\boldsymbol{t}}
\newcommand{\bs}{\boldsymbol{s}}

\newcommand{\dcut}{d_{\sss \Box}}

\newcommand{\cA}{\mathcal{A}}

\newcommand{\sW}{\mathscr{W}}

\newcommand{\ldp}{\textlcsc{LDP}}
\newcommand{\bd}{\bld{d}}

\newcommand{\tA}{\tilde{A}}
\newcommand{\tF}{\tilde{F}}
\newcommand{\tW}{\tilde{W}}
\newcommand{\tg}{\tilde{g}}
\newcommand{\tU}{\tilde{U}}
\newcommand{\tsW}{\tilde{\sW}}
\newcommand{\delcut}{\delta_{\sss \Box}}
\newcommand{\tPRd}{\tilde{\PR}_{n,d}}
\newcommand{\PRD}{\PR_{n,\bld{d}}}
\newcommand{\tPRD}{\tilde{\PR}_{n,\bld{d}}}

\newcommand{\tPRW}{\tilde{\PR}_{n,W_0^n}}
\newcommand{\bB}{\mathbb{B}_{\sss \Box}}
\newcommand{\rB}{\mathrm{B}_{\sss \Box}}

\newcommand{\tbB}{\tilde{\mathbb{B}}_{\sss \Box}}
\newcommand{\sM}{\mathscr{M}}
\newcommand{\degree}{\mathrm{deg}}

\newcommand{\sD}{\mathscr{D}}

\newtheorem{theorem}{Theorem}
\newtheorem*{claim*}{Claim}

\newtheorem{lemma}[theorem]{Lemma}
\newtheorem{proposition}[theorem]{Proposition}
\newtheorem{corollary}[theorem]{Corollary}
\newtheorem{assumption}{Assumption}
\newtheorem{remark}{Remark}
\newtheorem{fact}{Fact}
\newtheorem{defn}{Definition}

\let\plainqed\qedsymbol

\numberwithin{equation}{section}
\numberwithin{theorem}{section}

\title{Large deviation for uniform graphs with given degrees}
\author{Souvik Dhara$^{\dagger,\ddagger}$, Subhabrata Sen$^{\star}$ } 
\affil{$^\dagger$ Department of Mathematics, Massachusetts Institute of Technology, \\
$^\ddagger$ Microsoft Research, New England \\
$^\star$ Department of Statistics, Harvard University }
\date{\today}

\begin{document}

\maketitle
\begin{abstract}
 Consider the random graph sampled uniformly  from the set of all simple graphs with a  given degree sequence. Under mild conditions on the degrees, we establish a Large Deviation Principle ($\ldp$) for these random graphs, viewed as   elements of the graphon space. As a corollary of our result, we obtain $\ldp$s for functionals continuous with respect to the cut metric, and obtain an asymptotic enumeration formula for graphs with given degrees, subject to an additional constraint on the value of a continuous functional.   
Our assumptions on the degrees are identical to those of Chatterjee, Diaconis and Sly (2011), who derived the almost sure graphon limit for these random graphs. 
\end{abstract}
\blfootnote{Emails: \href{mailto:sdhara@mit.edu}{sdhara@mit.edu}, \href{mailto:subhabratasen@fas.harvard.edu}{subhabratasen@fas.harvard.edu}}
\section{Introduction}
In a seminal paper, Chatterjee and Varadhan \cite{CV11} initiated a study of large deviations for random graphs, and introduced a novel framework that   synergizes the classical theory of Large Deviations with the theory of dense graph limits (\citet{L2012}). 
 They embedded Erd\H{o}s-R\'enyi random graphs into the space of graphons, equipped with the cut-metric, and derived an $\ldp$ for the corresponding sequence of probability measures. 
As an important consequence, this
yields $\ldp$s for continuous functionals in the cut-metric topology, e.g. subgraph counts, largest eigenvalue, etc.
 Their result resolved a long-standing open question regarding large-deviations for sub-graph counts of dense  Erd\H{o}s-R\'enyi random graphs. This area has witnessed rapid developments subsequently
 --- 
we refer the interested reader to Chatterjee's Saint-Flour lecture notes~\cite{Cha17} for a detailed history of these problems and an elaborate description of recent breakthroughs.

 Numerous scientific applications naturally motivate the study of graphs with topological constraints, such as a fixed number of edges, triangles etc (see e.g. \cite{GKTB10,YW09,Orsini15}).
The desire to understand typical properties of constrained graphs motivates the study of random graphs, sampled uniformly, subject to these constraints. 
Natural examples include the Erd\H{o}s-R\'enyi uniform random graph with a constrained number of  edges, random regular graphs \cite{JLR00}, etc.
In statistical physics parlance, these  can be thought of as  microcanonical ensembles, whereas   unconstrained graphs, like Erd\H{o}s-R\'enyi, correspond to canonical ensembles \cite{HMRS18,SMHG15}. A rigorous study of constrained graphs often turns out to be extremely challenging--- in fact, even enumerating the total number of graphs, subject to combinatorial constraints, is  exceedingly non-trivial, and has attracted significant attention recently in Probability, Combinatorics, and Statistical Physics (see e.g. \cite{RS13,BH13,HLM19,Rad18,Yin15,RY13,KY17}).  The study of large deviations in this context is of natural interest -- indeed, this has deep, natural connections to the problem of counting graphs with atypical properties, subject to the topological constraints. Recently, Dembo and Lubetzky \cite{DL18} initiated a study of large deviations for constrained random graphs, and derived an $\ldp$ for  dense Erd\H{o}s-R\'enyi uniform random graphs, conditioned to have a fixed number of edges.

 In this paper, we study the uniform random graph with a  given degree sequence. The degrees are assumed to scale linearly in the number of vertices, so that we have a dense random graph. 
 Such random graphs are used extensively in Physics \cite{SMHG15} and Statistics \cite{BD11}, and have a rich history in Combinatorics \cite{B80,Wor99,BH13}. 
In general, this model is intractable to theoretical analysis.  In fact, characterizing the first order  asymptotics of simple functionals like triangle counts is  challenging in this case. In a breakthrough paper, 
Chatterjee, Diaconis and Sly \cite{CDS11} derived that, under fairly mild  conditions (see Assumption~\ref{assmp:degree}), these random graphs converge almost surely in the cut-metric, and identified the limit. 
Our main result, Theorem~\ref{thm:ldp-given-degree}, establishes an $\ldp$ for uniform random graphs under  identical conditions as \cite{CDS11}.  
This general theorem has two important corollaries. 
The first corollary (Corollary~\ref{cor:ldp-cont-func}) yields $\ldp$s for continuous functionals such as subgraph counts.
The second corollary (Corollary~\ref{cor:partition-function}) yields the convergence of the microcanonical partition function.
Further, it provides the asymptotic count of graphs with given degrees, subject to an additional constraint on the value of a continuous functional, in terms of a variational formula. 

Conceptually, the problem under consideration is significantly more challenging than the Erd\H{o}s-R\'{e}nyi case, due to the absence of edge-independence in these models. Further, in sharp contrast to the setting of \citet{DL18}, the number of degree constraints grows linearly with the number of vertices in the graph. To overcome this issue, we crucially exploit a deep idea put forth in \cite{CDS11}--- these random graphs may be sampled using appropriate inhomogeneous random graphs, conditioned to have the desired degrees (see Section~\ref{sec:CDS-facts}, and in particular \eqref{eq:graphon-l1-conv}). Unfortunately, even with access to this ingredient, one still faces substantial technical obstacles due to the inhomogeneity of the unconstrained model. Our proofs require a very delicate understanding of the cut-topology, and deviate significantly from the established techniques for the dense Erd\H{o}s-R\'{e}nyi model. To the best of our knowledge, this is the first instance where an $\ldp$ has been derived in inhomogeneous settings. Finally, we remark that requisite analytic properties of the candidate rate function, such as lower-semicontinuity, are not obvious here, and require careful analysis.

The rest of the paper is organised as follows: 
In Section~\ref{sec:defns},
we set up the framework necessary to state our main result. 
The statement of the main result and its corollaries is provided in Section~\ref{sec:main-results}.
In Section~\ref{sec:discussion}, we  discuss the relevant literature and collect some  open problems surfacing from our work.  
Section~\ref{sec:prelim} derives important analytic properties of the rate function.
In Section~\ref{sec:irg}, we prove a large deviation upper bound for inhomogeneous random graphs.
The proof of Theorem~\ref{thm:ldp-given-degree} is completed in Section~\ref{sec:degrees}.
Finally, we prove Corollaries~\ref{cor:ldp-cont-func}~and~\ref{cor:partition-function}  in Section~\ref{sec:proof-cor}.

\subsection{Definitions and concepts}\label{sec:defns}
\subsubsection{Graphons and the cut metric}\label{sec:defns-1}
A graphon is a measurable function $W:[0,1]^2\mapsto [0,1]$ that is symmetric, i.e., $W(x,y)=W(y,x)$ for all $x,y\in [0,1]$. 
To define the cut metric, let $\sM$ denote the set of all bijective, Lebesgue measure preserving maps $\phi:[0,1]\mapsto [0,1]$. 
The cut distance between two graphons $W_1$ and $W_2$ is given by 
\begin{eq}\label{eq:def-cut-distnace}
\dcut(W_1,W_2) = \sup_{S,T\subset [0,1]} \bigg|\int_{S\times T} \big(W_1(x,y) - W_{2}(x,y)\big) \dif x\dif y\bigg|,
\end{eq}
and the cut metric is given by 
\begin{eq}\label{eq:def-cut-metric}
\delcut (W_1,W_2)= \inf_{\phi\in \sM} \dcut (W_1,W_2^\phi)
\end{eq}
where $W^\phi (x,y) = W(\phi(x),\phi(y))$. See \cite[Lemma 3.5]{BCLSV08} for equivalent definitions of the cut metric.
Setting $\sW$ to denote the space of all graphons,
define the equivalence relation $W_1\sim W_2$ if $\delcut(W_1,W_2)=0$, and consider the quotient space $\tsW = \sW/ _{\sim} $. 
By \cite[Corollary 8.14]{L2012}, $W_1 \sim W_2$ if and only if $W_1^\phi = W_2^\psi$ for measure preserving transformations $\phi,\psi$.
Also, note that  $(\tsW,\delcut)$ is a compact metric space \cite[Theorem 9.23]{L2012}.
Henceforth, for any $W\in \sW$,  we always write $\tW$ to denote the equivalence class of $W$ in $\tsW$. 

\begin{defn}[Empirical graphon] \normalfont
For a graph $G_n$ with vertex set $[n] = \{1,\dots,n\}$ and edge set $E(G_n)$, the empirical graphon $W^{G_n}$ is given by 
\begin{eq}
W^{G_n}(x,y) = \begin{cases}
1 \qquad \text{ if }  (i,j)\in E(G_n),\ (x,y) \in \big[\frac{i-1}{n},\frac{i}{n}\big) \times \big[\frac{j-1}{n},\frac{j}{n}\big),\\
0 \qquad \text{ otherwise.}
\end{cases}
\end{eq}
\end{defn}
\begin{defn}[Graph Convergence] \normalfont
$(G_n)_{n\geq 1}$ is said to converge in $(\tsW,\delcut)$ if their empirical graphons converge.
\end{defn}
\begin{defn}[Subgraph densities] \label{defn:subgraph-dinsity} \normalfont
For a finite simple graph $H = (V(H),E(H))$ with $V(H) = [k]$, the subgraph density of $H$ in $W$ is defined as
\begin{eq}
t(H,W):= \int_{[0,1]^k} \prod_{(i,j)\in E(H)} W(x_i,x_j) \prod_{i=1}^k \dif x_i.
\end{eq}
Note that $t(H,W^\phi) = t(H,W^\psi)$ for measure preserving transformations $\phi,\psi$, and thus $t(H,\cdot)$ is well defined on $\tsW$. 
We write $t(H,\tW)$ to denote the subgraph density of $\tW\in\tsW$.
Also, \cite[Theorem 3.7]{BCLSV08} shows that $t(H,\cdot)$ is Lipschitz continuous on $(\tsW,\delcut)$ for any finite simple graph $H$.
\end{defn}

\begin{defn}[Degree distribution function] \normalfont
For any $W \in \sW$, the degree distribution function is defined by 
\begin{eq}
\degree_{W} (\lambda) = \Lambda\Big\{ x: \int_0^1W(x,y) \dif y \leq \lambda \Big\}, 
\end{eq}
where $\Lambda$ denotes the Lebesgue measure on $[0,1]$, and $\lambda \in[0,1]$. 
Observe that $\degree_W$ is well-defined on~$\tsW$. 
We write $\degree_{\tW}$ to denote the degree distribution function of $\tW \in \tsW$. 
\end{defn}

\begin{defn}[Graphons away from boundary]\label{defn:away-boundary} \normalfont
A graphon $W$ is said to be {\it away from boundary}  if there exists an $\eta>0$ such that $\eta < W(x,y)<1-\eta$. 
A sequence $(W_n)_{n\geq 1}$ is said to be away from boundary if for all $n\geq 1$, the above holds for some $\eta >0$ (independent of $n$).
\end{defn}
\subsubsection{Uniform graphs with given degrees} 
Consider a sequence of degree sequences $(\bld{d}^n)_{n\geq 1}$, $\bld{d}^n = (d_i^n)_{i\in [n]}$. 
Without loss of generality, we will assume that the degree sequence is non-increasing, i.e., $d_1^n\geq \dots \geq d_n^n$. For clarity of notation, we will simply write $d_i^n= d_i$, and $\bld{d}^n = \bld{d}$, and suppress the dependence of the degrees on $n$. 
Let $G_{n,\bd}$ denote the uniformly chosen random graph with degree sequence $\bld{d}$.

Of course, not all sequences $\bld{d}$ are valid degree sequences of simple graphs.
Such sequences are called graphical, and they are characterized by the celebrated Erd\H{o}s-Gallai theorem \cite{erdos1960graphen}. 
This theorem establishes that 
 $\bld{d}$ is graphical if and only if $\sum_{i\in [n]} d_i$ is even and for all $k\in [n]$ 
\begin{eq}\label{eq:erdos-gallai}
\sum_{i=1}^k d_i \leq k(k-1) + \sum_{i=k+1}^n \min\{d_i,k\}. 
\end{eq}
Thus $G_{n,\bd}$ is defined whenever  \eqref{eq:erdos-gallai} holds. 
Chatterjee, Diaconis and Sly \cite{CDS11} obtained the graphon limit of $G_{n,\bd}$ when the degrees converge, and the degree sequence lies in the interior of  an asymptotic  Erd\H{o}s-Gallai boundary  \eqref{eq:erdos-gallai}.
We state below the precise assumptions from~\cite{CDS11}, which will also be the underlying assumption for our large deviation result:
\begin{assumption} \label{assmp:degree} The degree sequence $\bld{d}^n$ satisfies the following:
\begin{enumerate}[(1)]
    \item There exists a non-increasing function $D:[0,1]\mapsto [0,1]$ such that 
    \begin{eq}
    \lim_{n\to\infty} \bigg(\Big|\frac{d_1}{n} -D(0)\Big|+\Big|\frac{d_n}{n} -D(1)\Big| +\frac{1}{n} \sum_{i=1}^{n}\Big|\frac{d_i}{n} -D\Big(\frac{i}{n}\Big)\Big|\bigg) = 0.
    \end{eq}
    \item There exist constants $0<c_1<c_2<1$ such that,$\forall x\in [0,1]$, $c_1\leq D(x) \leq c_2$, and 
    \begin{eq}
    \int_x^1 \min\{D(y),x\} \dif y + x^2 -\int_0^x D(y) \dif y >0.
    \end{eq}
\end{enumerate}
\end{assumption}
We write $\PRD$ to denote the probability measure on $\sW$ associated to the empirical graphon of $G_{n,\bd}$, and write $\tPRD$ to denote the corresponding push forward measure on $(\tsW,\delcut)$.  
The following was proved in \cite[Theorem 1.1]{CDS11}:
\begin{proposition}[{\cite[Theorem 1.1]{CDS11}}] \label{prop:graphon-limit}
Under {\rm Assumption~\ref{assmp:degree}}, almost surely $\bigotimes_{n\geq 1} \tPRD$, $(G_{n,\bd})_{n\geq 1}$ converges to the graphon $W_D$ in the cut-metric, as $n\to\infty$, where $W_{D}$ is given by 
\begin{eq}\label{limit:gen-degree}
W_D(x,y):= \frac{\e^{\beta(x)+\beta(y)}}{1+\e^{\beta(x)+\beta(y)}},
\end{eq}
where $\beta:[0,1]\mapsto \R$ is the unique function satisfying
$D(x) = \int_0^1\frac{\e^{\beta(x)+\beta(y)}}{1+\e^{\beta(x)+\beta(y)}} \dif y$, for all $x\in [0,1]$. 
\end{proposition}
Note that, under Assumption~\ref{assmp:degree}, $\|\beta\|_{\infty}<\infty$, which follows using \cite[Lemma 4.1]{CDS11}. Thus, $W_D$ is away from the boundary for any degree function $D$ satisfying Assumption~\ref{assmp:degree} in the sense of Definition~\ref{defn:away-boundary}. 

\subsection{Main results}\label{sec:main-results}
Our main result, Theorem~\ref{thm:ldp-given-degree}, stated below, derives a $\ldp$ for the sequence of probability measures $\tPRD$. To this end, for the convenience of the reader, we start with recalling the formal notion of a large deviation principle ($\ldp$). 
Let $\cX$ be a Polish space with Borel sigma-algebra $\mathscr{B}$. Let $I:\cX \to [0, \infty]$ be a lower semi-continuous function. 
A sequence of probability measures $(\PR_n)_{n \geq 1}$ on $(\cX, \mathscr{B})$  satisfies a large deviation principle ($\ldp$) with speed $s_n\nearrow \infty$ and good rate function $I$ if 
\begin{enumerate}
    \item[(i)] for all $\alpha\geq 0$, the level sets $\{x:I(x)\leq \alpha\}$ are compact,  
    \item[(ii)] for any closed set $F\subset \cX$ and open set $U\subset \cX$
\begin{eq}
\limsup_{n\to\infty} \frac{1}{s_n}\log\PR_n(F) \leq -\inf_{x\in F}I(x) \quad \text{and}\quad \liminf_{n\to\infty} \frac{1}{s_n}\log\PR_n(U) \geq -\inf_{x\in U}I(x).
\end{eq}
\end{enumerate}
Next, we introduce the candidate rate function in our context. 
For $W, W_0 \in \sW$ with $0<W_0<1$ a.s., we define 
\begin{eq} \label{eq:I-W0-W}
I_{W_0}(W) &= \frac{1}{2} \int_{[0,1]^2} \bigg[W(x,y) \log\Big(\frac{W(x,y)}{W_0(x,y)}\Big)+(1-W(x,y)) \log\Big(\frac{1-W(x,y)}{1-W_0(x,y)}\Big)\bigg] \dif x \dif y \\
& = \frac{1}{2} \sup_{a} \int_{[0,1]^2} \Big[a(x,y) W(x,y) - \log\big(W_0(x,y)\e^{a(x,y)}+1 - W_0(x,y)\big) \Big]\dif x \dif y  ,
\end{eq}where the supremum over $a$ in the final term ranges over all functions in $L^2([0,1]^2)$ satisfying $a(x,y)= a(y,x)$ for all $(x,y)\in [0,1]^2$  (for a proof of this variational characterization, see \cite[Lemma 5.2]{Cha17}). 
When $W$ takes values $0$ or $1$, we use the convention that $x\log x = 0$ to define the integrand in the first equality of \eqref{eq:I-W0-W}.
Unlike the rate function for the Erd\H{o}s-R\'enyi random graph in \cite[(7)]{CV11}, the function $I_{W_0}(\cdot)$ is not well-defined on the quotient space $\tsW$, i.e., $I_{W_0}(W_1)$ is not necessarily equal to $I_{W_0}(W_2)$, for $W_1\sim W_2$. 
To produce a valid candidate, we use the notion of a lower semi-continuous envelope.
Let $\bB(\tW,\eta) := \{g: \delcut(\tW,\tg) \leq \eta\}$, and define 
\begin{eq}\label{eq:rate-IRG}
J_{W_0} (\tW) = \sup_{\eta >0} \inf_{g \in \bB(\tW,\eta)} I_{W_0}(g).
\end{eq}
Note that $J_{W_0}$ is well-defined on $\tsW$.
Further, $J_{W_0}$ is lower semi-continuous on $(\tsW,\delcut)$ (see Lemma~\ref{lem:rate-lower-semicont}), i.e., the lower level sets $\{\tW: J_{W_{0}}(\tW)\leq \alpha\}$ are closed, and therefore compact due to the compactness of $(\tsW,\delcut)$. 
Thus, $J_{W_0}$ is a good rate function.

Next recall the definition of $W_D$ from \eqref{limit:gen-degree}.
The degree distribution function of $W_D$ is the inverse of $D$, i.e.,
\begin{eq}\label{lim-deg-dist}
\mu_D ([0,\lambda]) = \degree_{\tW_D}(\lambda) = \Lambda\{x: D(x) \leq \lambda\}.
\end{eq}
Define 
\begin{eq}\label{rate-constrained}
J_D(\tW) = 
\begin{cases}
J_{W_D}(\tW) &\quad \text{if }\degree_{\tW} (\lambda) = \mu_D([0,\lambda]),\   \forall\, \lambda\in [0,1], \\
\infty &\quad \text{otherwise}.
\end{cases}
\end{eq}
Using \cite[Theorem 2.16]{BCCG15} (see also \eqref{bound-prokhorov-metric} below), 
$\{\tW: \degree_{\tW} (\lambda) = \mu_D([0,\lambda]),\   \forall\, \lambda\in [0,1]\}$ is closed in $(\tsW,\delcut)$--- this establishes that
 $J_D$ is also a good rate function. Given this candidate rate function, we state our main result.
\begin{theorem} \label{thm:ldp-given-degree}
Under {\rm Assumption~\ref{assmp:degree}}, the sequence of probability measures $(\tPRD)_{n\geq 1}$ on $(\tsW, \delcut)$ satisfies a $\ldp$ with speed $n^2$ and  good rate function $J_D$ defined in \eqref{rate-constrained}. 
\end{theorem}
For the particular case of a random $d$-regular graph, Assumption~\ref{assmp:degree} holds when $d=\lfloor n p \rfloor$ for some $p\in (0,1)$ (see \cite[Remark 3]{CDS11}), and thus Proposition~\ref{prop:graphon-limit} and Theorem~\ref{thm:ldp-given-degree} are applicable. In this case, $W_{D}= p$, and we will show that the $\ldp$ rate function simplifies (see Lemma~\ref{lem:two-rate-functions-equal} for a proof).
Define 
\begin{eq}\label{eq:rate-regular}
J_p(\tW) = 
\begin{cases}I_p(W) \quad \text{if }\degree_{\tW} (\lambda) = \ind{p \leq \lambda}, \   \forall \lambda\in [0,1],\\
\infty \quad \text{otherwise}.
\end{cases}
\end{eq}
Note that $I_p(W_1) = I_p(W_2)$ for any $W_1\sim W_2$, thus $I_p$ is well-defined on $(\tsW,\delcut)$.
The following corollary states the corresponding $\ldp$ for the random regular graph. Let $p\in (0,1)$ and $d= \floor{np}$. Consider the degree sequence $\bld{d} = d\bld{1}$, and for this case simply denote the probability measure associated to the random regular graph by $\tPRd$. 
\begin{corollary}\label{thm:ldp-regular}
The sequence of probability measures $(\tPRd)_{n\geq 1}$ on $(\tsW, \delcut)$ satisfies a $\ldp$ with speed $n^2$ and good rate function $J_p$ defined in \eqref{eq:rate-regular}. 
\end{corollary}
As the main application of their $\ldp$, 
Chatterjee and Varadhan \cite{CV11} derived the $\ldp$s for subgraph counts of Erd\H{o}s-R\'enyi random graphs. 
Under the constraint on the number of edges, Dembo and Lubetzky \cite{DL18} also proved $\ldp$ results for subgraph counts. Below we state the corresponding results for~$G_{n,\bd}$.

Let $\tau: \tsW \mapsto \R$ be bounded and continuous with respect to $\delcut$. The $\ldp$ statement for~$\tau$ below will directly imply the $\ldp$ for subgraph counts of $G_{n,\bd}$, using the continuity of subgraph densities. 
Define the rate function
\begin{gather}
\phi_\tau(D,r) = \inf\{J_D(\tW): \tau (\tW) \geq r\}, \label{eq:var-formula-1} 
\end{gather}
Also, denote $\tsW_0=\{\tW\in \tsW: \degree_{\tW} \equiv \mu_D\}$ and  
\begin{eq}
l_\tau(D) = \tau (\tW_D), \quad r_{\tau}(D)=\sup\{r: \{\tW\in \tsW_0:\tau(\tW)\geq r\}\neq \varnothing\}.
\end{eq}
Let $\tau_{n,\bd}$ be the value of $\tau$ computed on the empirical graphon of $G_{n,\bd}$. Below we state the $\ldp$ result for $\tau_{n,\bd}$:
\begin{corollary}\label{cor:ldp-cont-func}
Let $\tau$ be a  bounded, continuous function on $(\tsW,\delcut)$. 
Then the following are true:
\begin{enumerate}[(1)]
    \item The function $\phi_{\tau}(D,\cdot)$ is left continuous, zero on $[0,l_\tau(D)]$, and finite, strictly positive on $(l_\tau(D),r_\tau(D)]$. 
    \item Let $r$ be any right continuity point of $\phi_{\tau}(D,\cdot)$. Then, under {\rm Assumption~\ref{assmp:degree}}, 
    \begin{eq}
    \lim_{n\to\infty} \frac{1}{n^2}\log \PR (\tau_{n,\bd}\geq r) = - \phi_{\tau}(D,r). 
    \end{eq}
    \item Let $F_{\star,r}$ be the set of minimizers in~\eqref{eq:var-formula-1}. Under {\rm Assumption~\ref{assmp:degree}}, for all $\varepsilon >0$, there exists $C = C(\varepsilon,\tau,D,r)>0$ such that 
    \begin{eq}
    \limsup_{n\to\infty} \frac{1}{n^2} \log  \PR\big(\delcut(W^{\sss G_{n,\bd}}, F_{\star,r}) \geq \varepsilon \big\vert \tau_{n,\bd}\geq r \big) \leq -C. \label{eq:conditional_graphon}
    \end{eq}
\end{enumerate}
\end{corollary}
Chatterjee and Diaconis \cite{CD13} used the $\ldp$ for  Erd\H{o}s-R\'enyi random graphs to evaluate the limit of the partition function associated with exponential random graphs \cite[Theorem 3.1]{CD13}. 
In a related direction, setting $\cG_{n,\bd}$ to be the set of all simple graphs on $n$ vertices with degree sequence $\bd$, 
 we consider  the probability measure on $\cG_{n,\bd}$ defined by 
\begin{eq}
\PR_{n,\bd,\tau} (G) = \e^{n^2(\tau (\tilde{W}^G) - Z_{n,\tau})}, \quad 
\forall G\in \cG_{n,\bd},
\end{eq}where
 $\tau$ is a bounded continuous function on $(\tsW,\delcut)$, and
$Z_{n,\tau} = \frac{1}{n^2} \log \sum_{G\in \cG_{n,\bd}} \e^{n^2\tau(\tilde{W}^G)}$. We will refer to  $Z_{n,\tau}$ as the microcanonical partition function. Its limiting value is naturally associated with the enumeration problem of graphs with given degrees and constrained sub-graph counts (see \eqref{eq:limit-number-graph} below). 
Our next corollary derives the limit of the microcanonical partition function. To this end, 
define the entropy function 
\begin{eq}\label{defn:entropy-single}
h_{e} (W) = \frac{1}{2} \int_{[0,1]^2} \Big[W(x,y) \log(W(x,y)) + (1-W(x,y)) \log(1-W(x,y))\Big]\dif x \dif y.
\end{eq}
Finally, let $N_{n.\tau}(\bld{d}, r)$ denote the number of graphs $G\in \cG_{n,\bd}$ with $\tau (\tW^{\sss G}) \geq r$. 

\begin{corollary} \label{cor:partition-function}
Let $\tau$ be a bounded continuous function on $(\tsW,\delcut)$. Under {\rm Assumption~\ref{assmp:degree}}, 
\begin{eq}\label{eq:limit-partition-function}
Z_{\tau} = \lim_{n\to\infty}Z_{n,\tau} = \sup_{\tW\in \tsW} \big(\tau(\tW)-J_D(\tW)\big) + h_e(W_D).
\end{eq}
Moreover, for any continuity point $r$ of $\phi_\tau (D,\cdot)$, 
\begin{eq}\label{eq:limit-number-graph}
\lim_{n\to\infty}\frac{1}{n^2} \log N_{n.\tau}(\bld{d}, r) = -\phi_{\tau} (D,r)+ h_e(W_D).
\end{eq}
\end{corollary}


\subsection{Discussion}\label{sec:discussion}

\paragraph{The variational problem.}
Corollary~\ref{cor:ldp-cont-func} characterizes the probability of a rare event in terms of a variational problem \eqref{eq:var-formula-1}.
From the perspective of large deviation theory, the natural follow up question concerns the  structure of $G_{n,\bd}$,  conditioned on the rare event. Using \eqref{eq:conditional_graphon},
this conditional structure corresponds to the minimizers
of~\eqref{eq:var-formula-1}. 
The variational problem \eqref{eq:var-formula-1} has attracted significant attention in the Erd\H{o}s-R\'{e}nyi case. 
For instance, it is now understood that in the so-called replica symmetric regime, conditioned on the upper tail event for triangle counts, the graph is close to an Erd\H{o}s-R\'enyi with a higher edge density \cite{LZ15}. Note that the replica symmetric regime is no longer tenable under exact constraints, such as a fixed number of edges, triangles, degrees, etc. In a set of related papers, 
\cite{RS13,KRRS17,KRRS17b,KRSR16} study the structure of the minimizer under constraints on the edge, triangle or star counts, and discover intriguing characteristics of the minimizers. 
%
However, to the best of our knowledge, this problem has not been studied under degree constraints. We expect this case to be considerably more challenging than the prior settings.   


A careful reader has  noticed that Corollary~\ref{cor:ldp-cont-func}~(2) holds when $r$ is a continuity point of $\phi_{\tau} (D, \cdot)$. 
For Erd\H{o}s-R\'enyi random graphs, the continuity of this function has been established, when $\tau$ represents a subgraph density, the largest eigenvalue, etc. \cite{LZ15,Cha17}. Their proof is perturbative, and the idea does not generalize to the setting with given degrees. 
%
In fact, $\phi_{\tau}(D,\cdot)$ could be degenerate in constrained spaces.
For example, the largest eigenvalue of random $d$-regular graphs equals $d$, and thus the rate function is degenerate. More generally, a deterministic function of the degrees, e.g. any $k$-star density, is constant in this case, and gives rise to degenerate rate functions. 

\paragraph{Counting graphs with given degrees and subgraph densities.}
Counting graphs with given degrees has been studied extensively in Combinatorics  \cite{MW90,LW17,BH13,HLM19}. 
 For example, \cite[Theorem 1.4]{BH13} evaluates the leading asymptotics of the number of graphs with given degrees, and expresses it in terms of an entropy. Corollary~\ref{cor:partition-function} yields a formula for the asymptotic number of graphs with given degrees and a  specified  subgraph count. However, this description is completely implicit, and explicit solutions for general degree sequences could be significantly challenging. 
 






\paragraph{The sparse regime.}
The breakthrough result of \citet{CV11} completely resolved the question of large deviations for subgraph counts of dense Erd\H{o}s-R\'{e}nyi random graphs. 
The corresponding question for sparse Erd\H{o}s-R\'{e}nyi random graphs $G(n,p)$ with $p \to 0$  has intrigued researchers in Probability and Combinatorics for a long time.
For any fixed graph $H$ and $\delta>0$, the infamous upper tail problem sought to understand the probability that the number of copies of $H$ in $G(n,p)$ exceeds $(1+\delta)$ times its expectation. 
Perhaps surprisingly, it is even difficult to come up with a good general guess as to what the correct order of the exponential rate of decay is. 
This can be observed in a class of counter-examples to the DeMarco-Kahn upper tail conjecture, constructed by \u{S}ileikis and Warnke \cite{SW19}. 
To address this challenging  question, Chatterjee and Dembo \cite{CD16} initiated the theory of non-linear large deviations. They establish that for any fixed subgraph $H$ and $\delta>0$, 
the upper tail probability reduces  to a variational problem on the space of weighted graphs whenever $p\to0$, $p\geq n^{-\alpha_H}$.  Remarkably, the variational problem was solved in the special case where $H$ is a clique by \citet{LZ17} shortly thereafter. Subsequently, \citet{BGLZ17} resolved this question for all fixed subgraphs.
Following the initial breakthrough of \citet{CD16}, the exponent $\alpha_H$ was improved considerably by \citet{Eld18}.
Recently, \citet{CD18}, \citet{Aug18}, and \citet{HMS19} have further improved the bounds on $\alpha_H$, deriving the optimal exponent for certain specific subgraphs such as cycles, cliques, regular graphs etc. 

These exciting recent developments have dramatically improved our understanding of the upper tail problem on sparse Erd\H{o}s-R\'{e}nyi random graphs. 
It would be fascinating to answer this question for sparse random graphs with a given degree sequence.
In fact, the simpler question of enumeration of all  graphs with a given degree sequence is not very well understood at present. We  believe these questions furnish a fertile ground for future research. 
After the first version of this paper was posted online, there have been recent interesting developments for sparse $d$-regular random graphs. 
 For $n^{1-\varepsilon(H)}\ll d\ll n$ (where $\varepsilon(H)$ is explicit), 
 Bhattacharya and Dembo~\cite{BD20} resolved the upper tail problem for subgraphs $H$ having a regular two-core.  
 Recently, Gunby~\cite{Gun20} considered general subgraphs $H$, solving the upper tail problem when some subgraph $H'\subset H$ has average degree greater than 4.

\section{Properties of the rate function} \label{sec:prelim}
Recall the definition of $J_{W_0}$ from \eqref{eq:rate-IRG}. In this section, we will prove some elementary facts about $J_{W_0}$, that will be crucial in our proofs.
Throughout, we denote $\tbB(\tW,\eta) = \{\tW ': \delcut(\tW,\tW') \leq \eta\}$ and  $\bB(\tW,\eta) = \{W': \tW'\in \tbB(\tW,\eta)\}$.
We first prove the lower semi-continuity of our rate function.
\begin{lemma}\label{lem:rate-lower-semicont}
The function $J_{W_0} (\tilde{W}) = \sup_{\eta >0} \inf_{W' \in \bB(\tW,\eta)} I_{W_0}(W')$ is well-defined on the space~$\tsW$.
Moreover, $J_{W_0}$ is lower semi-continuous on $(\tsW,\delcut)$.
\end{lemma}
\begin{proof}
For any $W_1\sim W_2$, it follows that $\{W': \delcut(\tW_1,\tW') \leq \delta\} = \{W': \delcut(\tW_2,\tW') \leq \delta\}$, and therefore $J_{W_0}$ is well-defined on $\tsW$. 
Define the function $H:\tsW\mapsto [0,\infty]$ by $ H(\tW' ) = \inf_{g: \delcut (\tW',\tg) = 0} I_{W_0}(g)$.
Now, 
\begin{eq}
J_{W_0}(\tW) = \sup_{\eta >0} \inf_{W' \in \bB(\tW,\eta)} I_{W_0}(W') = \sup_{\eta >0} \inf_{\tW' \in \tbB(\tW,\eta)} H(\tW') = \liminf_{\tW'\to \tW} H(\tW'),
\end{eq}
and it is a standard fact in analysis that the function obtained by taking pointwise $\liminf$ of a function must be lower semi-continuous.
This completes the proof.
\end{proof}
The next result shows that the relative entropy between $W$ and $W_0$ is zero if and only if they are in the same equivalence class.

\begin{lemma} \label{lem:rate-zero}
$J_{W_0} (\tW) = 0$ if and only if $\delcut(\tW,\tW_0) = 0$.
\end{lemma}
\begin{proof}
The sufficiency part is obvious. 
 To see the necessity, assume $J_{W_0}(\tW) = 0$. 
 In this case, there exists $(g_n)_{n\geq 1} \subset \sW$ with $\delcut(\tg_n,\tW) \to 0$ such that $I_{W_0}(g_n) \to 0$. Using Taylor expansion, one immediately obtains
 \begin{eq}
& \frac{1}{2} \int_{[0,1]^2} \bigg[ g_n(x,y) \log \Big( \frac{g_n(x,y)}{W_0(x,y)}\Big) + (1- g_n(x,y)) \log \Big(\frac{1- g_n(x,y)}{1-W_0(x,y)} \Big)\bigg] \mathrm{d}x \mathrm{d}y \\
 &\geq \int_{[0,1]^2} (g_n(x,y) - W_0(x,y))^2 \mathrm{d}x \mathrm{d}y.  
 \end{eq}
Next, using Cauchy-Schwarz inequality, $\|g_n-W_0\|_{\sss L_1} \to 0$, and consequently $\delcut(\tg_n,\tW_0)\to 0$. Thus, $\delcut(\tW, \tW_0)\leq \delcut(\tg_n,\tW)+ \delcut(\tg_n,\tW_0)\to 0$ as $n\to\infty$. This completes the proof. 
\end{proof}
Next we will prove that if we have a sequence $W_0^n$ converging to $W_0$ in $L_1$,  then the corresponding rate functions converge as well. 

\begin{lemma}\label{lem:l1-conv-vont}
Suppose that $\|W_0^n-W_0\|_{\sss L_1} \to 0$, and that $(W_0^n)_{n\geq 1}$ is away from  boundary.
Then, $J_{W_0^n}(\tilde{W}) \to J_{W_0}(\tilde{W})$ uniformly over $W\in \sW$, as $n \to \infty$.
\end{lemma}

\begin{proof}
Let $\eta>0$ be such that $\eta < W_0^n <1-\eta$ for all $n\geq 1$. 
By taking limit as $n\to \infty$, we also have that  $\eta < W_0 <1-\eta$ almost surely.
Thus, using the Lipschitz continuity of the log function, it follows that for all $x,y$, 
\begin{eq}
\max \Big\{ \Big| \log \Big(\frac{W_0^n(x,y)}{W_0(x,y)}\Big) \Big|,  \Big|\log \Big(\frac{1-W_0^n(x,y)}{1-W_0(x,y)}\Big) \Big| \Big\} \leq c |W_0^n(x,y) - W_0(x,y)|,
\end{eq}for some constant $c>0$.
Now, 
\begin{eq} \label{eq:bound-entrop-difference}
&|I_{W_0^n}(W)- I_{W_0}(W)| \\
&= \bigg| \int_{[0,1]^2} W(x,y)\log \Big(\frac{W_0^n(x,y)}{W_0(x,y)}\Big) + (1-W(x,y))\log \Big(\frac{1-W_0^n(x,y)}{1-W_0(x,y)}\Big)\dif x \dif y  \bigg| \\
& \leq c \int_{[0,1]^2} |W_0^n(x,y) - W_0(x,y)| \dif x \dif y = c \|W_0^n-W_0\|_{\sss L_1}.
\end{eq}
The proof now follows upon using the definition of the rate function, and noting the bound in the final term of \eqref{eq:bound-entrop-difference} is uniform over $W\in\sW$. 
\end{proof}
We finally conclude this section by showing that for the special case of random regular graphs, the relative entropy reduces to the form give in \eqref{eq:rate-regular}.

\begin{lemma}\label{lem:two-rate-functions-equal}
Fix $W\in \sW$ and  $p\in(0,1)$.
Then, $J_{p}(\tW) = \sup_{\eta >0} \inf_{g \in \bB(\tW,\eta)} I_p (g) = I_p(W)$.
\end{lemma}
\begin{proof}
By \cite[Corollary 5.1]{Cha17}, whenever $\dcut (W_n,W) \to 0$, we have 
\begin{eq}\label{lower-semi-cont-cut}
\liminf_{n\to\infty}I_{p}(W_n) \geq I_{p}(W).
\end{eq}
Now let us denote $I (\eta) := \inf_{ g\in \bB(\tW,\eta)} I_{p}(g)$. 
Then $I(0) = \inf_{g: \delcut (\tg,\tW) = 0} I_p(g) = I_p(W)$, since $I_p(W_1) = I_p(W_2)$ whenever $W_1\sim W_2$. 
%
Also, $I(\eta) \leq I(0)$. 
Thus, in order to complete the proof, we need to show that $I(0) = \sup_{\eta>0} I(\eta)$, i.e., for all $\varepsilon >0$, $\exists \eta(\varepsilon)>0$ such that $I(\eta) > I(0) - \varepsilon$ for all $\eta\in (0,\eta(\varepsilon))$. 
Suppose that this does not hold. Using \cite[(3.15)]{BCLSV08}, there exists $\varepsilon>0$ and $\eta_{n}\to 0$ such that $I(\eta_n) \leq I(0) -\varepsilon$ for all $n\geq 1$.
This implies that there exist $(g_n)_{n\geq 1} \subset \sW$ and  $(\phi_n)_{n\geq 1} \subset \sM$  such that $\dcut (W,g_n^{\phi_n}) \leq  \eta_n$, but $I_{p}(g_n) < I(\eta_n) + \varepsilon/2< I(0) -\varepsilon/2$ for all $n\geq 1$. 
Since $I_{p}(g_n) = I_{p}(g_n^{\phi_n})$, it follows that $I_{p}(g_n^{\phi_n})< I(0) -\varepsilon/2$.
Now, using \eqref{lower-semi-cont-cut}, we have that $I_{p} (W) < I(0)-\varepsilon/2$ which yields a contradiction because $I(0) = I_{p}(W)$.
\end{proof}

\begin{remark}\normalfont
In a recent preprint, Markering~\cite{Mar20} derives a tractable form for the lower semi-continuous envelope $J_{W_0}(\cdot)$ by showing that 
$\sup_{\eta >0} \inf_{g \in \bB(\tW,\eta)} I_{W_0} (g) = \inf_{\phi\in \sM} I_{W_0} (g^\phi)$ for any $W_0$ such that $\log W_0, \log (1-W_0)\in L_1$. 
\end{remark}

\section{An upper bound for inhomogeneous random graphs}\label{sec:irg}
In this section, we obtain a large deviation upper bound for inhomogeneous random graphs. 
Let $\sW^{\sss (r)}\subset \sW$ denote the space of block constant graphons with $r$ equal-sized blocks, i.e., for any $g\in \sW^{\sss (r)}$, we have  
$g(x,y) = g_{ij}$ for all $x,y\in [\frac{i-1}{r},\frac{i}{r})\times [\frac{j-1}{r},\frac{j}{r})$.
To generate inhomogeneous random graphs on $n$ vertices, we take $g\in \sW^{\sss (n)}$ of the following special form with zeroes on the diagonal: 
\begin{eq}\label{eq:piecewise-const}
g (x,y) = 
\begin{cases}
g_{ij}, \quad &x,y \in \big[\frac{i-1}{n},\frac{i}{n}\big)\times \big[\frac{j-1}{n},\frac{j}{n}\big), \quad i\neq j\\
0 &\text{otherwise}.
\end{cases}
\end{eq}
We denote the collection of graphons in \eqref{eq:piecewise-const} by $\sW^{\sss (n)}_{\sss \mathrm{IRG}}$.
Given any graphon $W \in \sW$, consider the random graph $G_{n}= G_n(W)$ on vertex set $[n]$ obtained by keeping an edge between vertices $i$ and $j$ with probability $ W((i-1)/n,(j-1)/n)$.
Let $\PR_{n,W}$ denote the probability measure on $\sW$ induced by the empirical graphon of $G_{n}(W)$, and let $\tilde{\PR}_{n,W}$ denote the corresponding measure on $\tsW$.
The following proposition derives the LDP upper bound for $\PR_{n,W_0^n}$ where $W_0^n\in \sW^{\sss (n)}_{\sss \mathrm{IRG}}$.  
Recall $\tbB(\tW,\eta) = \{\tW ': \delcut(\tW,\tW') \leq \eta\}$ and  $\bB(\tW,\eta) = \{W': \tW'\in \tbB(\tW,\eta)\}$.
For any $W_0^n\in \sW^{\sss (n)}_{\sss \mathrm{IRG}}$, the value in the diagonal blocks  $[\frac{i-1}{n},\frac{i}{n})\times [\frac{i-1}{n},\frac{i}{n})$ is zero. 
Nevertheless, we say that $(W_0^n)_{n\geq 1}$ with $W_0^n\in \sW^{\sss (n)}_{\sss \mathrm{IRG}}$ is away from the boundary if there exists some fixed $\eta>0$ such that $\eta <W_0^n(x,y)<1-\eta$ in the non-diagonal blocks for all $n\geq 1$.
\begin{proposition} \label{thm:ldp-IRG}
Fix $\varepsilon>0$.  
Let $W_0^n \in \sW^n_{\sss \mathrm{IRG}}$ be such that $\|W_0^n -W_0\|_{\sss L_1} \to 0$, and further assume that $(W_0^n)_{n\geq 1}$ is away from boundary. Then, there exists $\eta(\varepsilon)>0$ such that for all $\eta \in (0,\eta(\varepsilon))$
\begin{gather}
\limsup_{n\to\infty} \frac{1}{n^2}\log\tPRW(\tbB(\tW,\eta)) \leq - \inf_{f\in \bB(\tW,4\varepsilon) } I_{W_0} (f) +\varepsilon. \label{ldp-reg-up-IRG}
\end{gather}
\end{proposition}

\begin{remark} \normalfont 
Proposition~\ref{thm:ldp-IRG} proves $\ldp$ upper bound for inhomogeneous random graphs under the stated conditions. A matching lower bound can be derived following the  arguments of~\cite{CV11}, which shows that $(\tilde{\PR}_{n,W_0^n})_{n\geq 1}$ satisfies $\ldp$ with speed $n^2$ and rate function~$J_{W_0}$. 
For the constrained case, additional challenges arise in the proof of the lower bound which we deal with in Section~\ref{sec:ldp-lower-bound}.
\end{remark}

Let $\sM_n$ be the set of permutations of $[n]$.
For $\sigma_n\in\sM_n$, let $G_n^{\sigma_n}$ denote the graph with vertices relabelled according to the permutation $\sigma_n$. 
In the special case of Erd\H{o}s-R\'enyi random graphs with $W_0^n \equiv p$, the distribution of $G_n^{\sigma_n}$ is the same for all $\sigma_n\in\sM_n$. 
This is a crucial ingredient in the LDP upper bound proof of \citet{CV11}, since the cut-metric also optimizes over all relabellings (see \cite[Lemma 2.5]{CV11}). 
For general $W_0^n$, the distribution of $G_n^{\sigma_n}$ depends on  $\sigma_n\in\sM_n$, and one needs to optimize the upper bound over all the $n!$ relabellings, which grows with $n$.
The argument for Erd\H{o}s-R\'enyi random graph does not generalize for such an optimal relabelling. 
To this end, we proceed in two steps: 
\begin{enumerate}[(i)]
    \item[(S1)] We replace $W_0^n$ by a block constant graphon $g_r\in \sW^{\sss (r)}$ with fixed number of blocks that is ``close'' to $W_0^n$. The error due to such an operation is small when $r$ is large, as we prove in Lemma~\ref{lem:block-approx-ldp-scale}.
    \item[(S2)] The next step is the key conceptual ingredient. If the base graphon $g_r$ is a block constant,  we can restrict ourselves to a finite number of relabellings without incurring significant error. Thus we only need to optimize over this finite set. We prove this in Lemma~\ref{lem:relabelling-uniform}.
\end{enumerate}
We formalize (S1) and (S2) in Sections~\ref{sec:block-const-approx} and \ref{sec:relabelling} respectively. Finally, we complete the proof of Proposition~\ref{thm:ldp-IRG} in Section~\ref{sec:proof-IRG-LDP-UB}.

\subsection{Replacing base graphon by block constants} \label{sec:block-const-approx}
The following statement allows us to replace $W_0^n$ by a block constant graphon with fixed number of blocks in our LDP upper bound.
\begin{lemma} \label{lem:block-approx-ldp-scale}
Let $W_0^n\in \sW^{\sss (n)}_{\sss \mathrm{IRG}}$ be such that $\|W_0^n-W_0\|_{\sss L_1} \to 0$, and $(W_0^n)_{n\geq 1}$ is away from boundary.  
There exists $(g_r)_{r\geq 1} \subset \sW$ that is away from boundary such that $g_r\in \sW^{\sss (r)}$, and for all $\varepsilon >0$ (sufficiently small), there exists $N_0=N_0(\varepsilon)$ such that for all $n\geq r\geq N_0$,  $W\in\sW$ and $\eta >0$
\begin{eq}
\bigg|\frac{1}{n^2} \log \PR_{n,W_{0}^n}( \bB(\tW,\eta)) - \frac{1}{n^2} \log \PR_{n,g_r}(\bB(\tW,\eta))\bigg| < \varepsilon.
\end{eq}
\end{lemma}
\begin{proof}
Define, for all  $(x,y) \in [\frac{i-1}{r},\frac{i}{r}) \times[\frac{j-1}{r},\frac{j}{r})$, 
\begin{eq}
g_{r}(x,y) = g_{ij}= r^2 \int_{\big[\frac{i-1}{r},\frac{i}{r}\big) \times\big[\frac{j-1}{r},\frac{j}{r}\big)} W_0(u,v)\dif u\dif v.
\end{eq}
Using \cite[Proposition 2.6]{Cha17},  $\|g_r - W_0\|_{\sss L_1} \to 0$ as $r\to\infty$,  and thus it follows that for all $\varepsilon>0$, there exists $N_0$ such that, for all $r,n\geq N_0$, $\|g_r - W_0^n\|_{\sss L_1} < \varepsilon$.
Also, since $W_0$ is away from the boundary, so is $(g_r)_{r\geq 1}$.
Now, note that 
\begin{eq}\label{eq:upper-equal}
\PR_{n,W_{0}^n}(\bB(\tW,\eta)) = \int_{\bB(\tW,\eta)} \e^{\log \frac{\dif \PR_{n,W_0^n}}{\dif \PR_{n,g_r}}} \dif \PR_{n,g_r}.
\end{eq}
Let $(w_{uv})_{1\leq u<v\leq n}$ be the block constant values of $W_0^n$.
Thus, for $I_{uv} \in \{0,1\}$, and $n\geq r$, 
\begin{eq}\label{eq:derivaitve}
&\log \bigg[\frac{\dif \PR_{n,W_{0}^n}}{\dif \PR_{n,g_r}}(I_{uv})_{u<v} \bigg]\\
&= \sum_{1\leq i\leq j\leq r} \sum_{\stackrel{u<v}{\frac{u-1}{n}\in [\frac{i-1}{r},\frac{i}{r}),\frac{v-1}{n}\in [\frac{j-1}{r},\frac{j}{r})}} \bigg(I_{uv} \log\Big(\frac{w_{uv}}{g_{ij}}\Big)+(1-I_{uv}) \log\Big(\frac{1-w_{uv}}{1-g_{ij}}\Big)\bigg).
\end{eq}
Thus, for any $(I_{uv})_{u<v}$,
\begin{eq}
&\frac{1}{n^2}\bigg|\log \frac{\dif \PR_{n,W_{0}^n}}{\dif \PR_{n,g_r}}(I_{uv})_{u<v}\bigg|\\
&\leq \frac{1}{n^2}\sum_{1\leq i\leq j\leq r} \sum_{\stackrel{u<v}{\frac{u-1}{n}\in [\frac{i-1}{r},\frac{i}{r}),\frac{v-1}{n}\in [\frac{j-1}{r},\frac{j}{r})}}  \bigg( \bigg|\log\Big(\frac{w_{uv}}{g_{ij}}\Big)\bigg|+ \bigg|\log\Big(\frac{1-w_{uv}}{1-g_{ij}}\Big)\bigg|\bigg) \\ 
&\leq C\|W_0^n-g_r \|_{\sss L_1} < C\varepsilon,
\end{eq}for some constant $C>0$, and for all $n\geq r\geq N_0$, where in the last step we have used the Lipschitz continuity of $\log$ on $[c_1,c_2]$ with $0<c_1<c_2<\infty$, and the fact that $(g_{r})_{r\geq 1}$ and $(W^n_0)_{n\geq 1}$ are away from the boundary. 
Now, \eqref{eq:upper-equal} yields that 
\begin{eq}
\PR_{n,W_{0}^n}(\bB(\tW,\eta)) \leq \e^{C \varepsilon n^2}\PR_{n,g_r}(\bB(\tW,\eta)).
\end{eq}
Thus the proof follows by replacing $C\varepsilon$ by $\varepsilon$. 
\end{proof}

\subsection{Approximation of relabelled graphs} \label{sec:relabelling}
Recall that $G_n^{\sigma_n}$ is obtained from the graph $G_n$ by relabelling the vertices with the permutation $\sigma_n\in\sM_n$. 
The next result shows that, for all large enough~$n$, we can construct a finite set of relabellings which can be used to approximate the distributions of $G_n^{\sigma_n}$ for all $\sigma_n\in\sM_n$. Recall the definition of $\sM$ from Section~\ref{sec:defns-1}. 
\begin{lemma}\label{lem:relabelling-uniform}
Suppose that $W_0,W\in \sW^{\sss (r)}$ with $r\geq 1$. 
Then, for any $\varepsilon>0$, there exists $n_{0} = n_{0}(r,\varepsilon)$, and a finite set $\T = \T (r,\varepsilon) \subset \sM$ such that for all $n\geq n_0$ and $\sigma_n\in \sM_n$, there exists $\tau \in\T$ satisfying
\begin{eq}\label{eq:fixed-perm-bound}
\PR_{n,W_0} (\dcut(W^{\sss G_n^{\sigma_n}}, W)\leq \varepsilon) \leq \PR_{n,W_0} (\dcut(W^{\sss G_n,\tau }, W )\leq 2\varepsilon).
\end{eq}
\end{lemma}
\begin{proof}
We write $A_i = [\frac{i-1}{r},\frac{i}{r})$ for $i\in [r]$. 
For a vertex $v\in [n]$, we say that $v$ is in the interval $A\subset [0,1]$, denoted by $v\rightsquigarrow A$, if $[\frac{v-1}{n},\frac{v}{n}) \subset A$.
Without loss of generality, we take $n\geq r$, 
so that any vertex can be in at most one $A_i$.
Let $C_{ij} (\sigma_n) = \{v: v\rightsquigarrow A_i, \sigma_n(v)\rightsquigarrow A_j\}$. 
Thus, if we think of $v\rightsquigarrow A_i$ as a vertex of type $i$, then $|C_{ij} (\sigma_n)|$ counts the number of type $i$ vertices that get mapped into $A_j$ under the permutation $\sigma_n$.
The basic idea of the proof is that since $W_0$ and $W$ are block constants, the distribution of $\dcut(W^{\sss G_n^{\sigma_n}}, W)$ and $\dcut(W^{\sss G_n^{\tau_n}}, W)$ remains the same if $|C_{ij} (\sigma_n)|=|C_{ij} (\tau_n)|$ for all $i,j$. 
Thus, if $\tau\in \sM$ be such that the number of type-$i$ vertices that get mapped to block $j$ under $\tau$ is approximately $|C_{ij} (\sigma_n)|$, then distributions of  $\dcut(W^{\sss G_n^{\sigma_n}}, W)$ and $\dcut(W^{\sss G_n,\tau}, W)$ are approximately close. Below, we make this intuition precise.

Fix $\bt\in T$, 
where
\begin{eq}
T:=\bigg\{(t_{ij})_{i\in [r], j\in [s]}:t_{ij}\in (0,1), \sum_{i\in [r]}t_{ij} = \frac{1}{r},\sum_{j\in [r]}t_{ij} = \frac{1}{r}\bigg\},
\end{eq}
and let
\begin{eq}
\sM_n(\bt,\eta) = \Big\{\sigma_n\in \sM_n: \frac{|C_{ij} (\sigma_n)|}{n}\in (t_{ij}-\eta,t_{ij} +\eta),\ \forall i,j\in [r], 
\Big\}.
\end{eq}
Thus, $\sM_n(\bt,\eta)$ identifies the class of permutations under which $A_j$ consists roughly of $nt_{ij}$ many type-$i$ vertices (when $\eta$ is small). 
Also, we write 
\begin{eq}
A_{ij} = \bigg[\frac{i-1}{r} +  \sum_{k=1}^{j-1}t_{ik}, \frac{i-1}{r} +  \sum_{k=1}^{j}t_{ik}\bigg).
\end{eq}
Now, consider $\tau\in \sM$ satisfying
\begin{eq}\label{eq:MP-property}
\tau (A_{ij})= A_{ji}, \quad \forall i,j\in [r]. 
\end{eq}
More precisely, we take $\tau $ to be $\tau(x) = c_{ij} +x$ for $x\in A_{ij}$, where $c_{ij}$'s are chosen so that~\eqref{eq:MP-property} is satisfied. 
The map $\tau$ can be understood as follows. The interval $A_{ij}$ contains roughly $nt_{ij}$ many type-$i$ vertices, 
which are the only type-$i$ vertices to get mapped to the interval $A_{j}$. Thus, under $\tau$, $A_j$ contains roughly $nt_{ij}$ many type-$i$ vertices. 
Note also that after $\tau$ has been applied, the labels of vertices of type-$i$ inside each block are ``sorted'' in increasing order.

Next, we claim that, for any $\varepsilon>0$, there exists $\eta=\eta(\varepsilon)>0$ (independent of $\bt$) such that for any $\eta \in (0,\eta (\varepsilon))$, and $\sigma_n \in \sM_n(\bt,\eta)$, there exists a coupling between $\dcut(W^{\sss G_n^{\sigma_n}},W)$ and $\dcut(W^{\sss G_n,\tau},W)$ such that 
\begin{eq}\label{eq:coupling-MP}
\lim_{n\to\infty} \PR\big(\big|\dcut(W^{\sss G_n^{\sigma_n}},W) - \dcut(W^{\sss G_n,\tau},W)\big|>\varepsilon\big) = 0.
\end{eq}
The proof of \eqref{eq:coupling-MP} goes as follows:  Given any composition $\bt\in T$, we choose a ``sorted'' measurable bijection $\tau$ given by \eqref{eq:MP-property}. 
Then we fix $\sigma_n$ which has approximate composition~$\bt$. We re-arrange so that it is also in sorted form within blocks. Finally, we couple these sorted models. 

We write $\tau(v) \rightsquigarrow A_j$ if $\tau([\frac{v-1}{n},\frac{v}{n})) \subset A_j$. 
Let $C_{ij} (\tau) = \{v: v\rightsquigarrow A_i, \tau(v)\rightsquigarrow A_j\}$.
By construction, $||C_{ij} (\tau)| - n t_{ij}|\leq 1$. Also, for any $\sigma_n\in \sM_n(\bt,\eta)$, $||C_{ij} (\sigma_n)\blue{|} - nt_{ij}| \leq 2\eta n$.
Let $n_{ij} = \min\{C_{ij} (\tau),C_{ij} (\sigma_n)\}$.
Thus $A_j$ contains at least $n_{ij}$ many type-$i$ vertices, both under $\sigma_n$ and~$\tau$. 
Let $\sigma_n^0\in \sM_n$ be such that  $\sigma_n^0$ permutes vertices within blocks $A_j$ only, and $\sigma_n^0$ sorts the different types of vertices within blocks in ascending order. 
More formally, $\sigma_n^0$ satisfies, 
\begin{enumerate}[(1)]
    \item For $\sigma_n(u)\rightsquigarrow A_j$, we have $\sigma_n^0\circ \sigma_n(u)\rightsquigarrow A_j$.
    \item For $u\rightsquigarrow A_{i_1}$ and $v\rightsquigarrow A_{i_2}$ with $i_1<i_2$, we have $\sigma_n^0\circ \sigma_n(u)<\sigma_n^0\circ \sigma_n(v)$.
    \item For $u,v \rightsquigarrow A_i$ with $u<v$, we have $\sigma_n^0\circ \sigma_n(u)<\sigma_n^0\circ \sigma_n(v)$.
\end{enumerate}
Now we couple the edges between $n_{ij}$ many vertices between the blocks. 
More precisely, 
$W^{\sss G_n^{\sigma_n}}$ and $W^{\sss G_n, \tau}$ are coupled such that there is an edge between $\sigma_n^0\circ \sigma_n(u)$ and  $\sigma_n^0\circ \sigma_n(v)$ if and only if $W^{\sss G_n, \tau}$ takes value 1 on $\tau([\frac{u-1}{n},\frac{u}{n}))\times \tau([\frac{v-1}{n},\frac{v}{n}))$.
This indeed gives a coupling because an application of permutations such as $\sigma^0_n$ which only permutes the vertices within blocks, does not change the distribution of $\dcut (W^{\sss G_n^{\sigma_n}},W)$.
Note that this coupling does not specify the edges incident to at most $2\eta n + 2$ many vertices. 
This can cause an error of at most $3\eta $ in $L_1$-norm, and hence an error of at most $3\eta $ in the cut-norm. 
Taking $\eta (\varepsilon) = \varepsilon/3$, the proof of \eqref{eq:coupling-MP} follows.

Finally, consider any finite set $(\bld{t}_\alpha)_{\alpha} \subset T$ such that for any $\bld{s}\in T$, there exists $\alpha$ with $\|\bs - \bt_\alpha\|_{\infty} < \eta$.
%
The proof follows by choosing a $\tau$ satisfying \eqref{eq:coupling-MP} for each $\bt_\alpha$.

\end{proof}

\subsection{Proof of Proposition~\ref{thm:ldp-IRG}}\label{sec:proof-IRG-LDP-UB}
Fix $\varepsilon >0$. Recall the setup of Proposition~\ref{thm:ldp-IRG}. Using Lemma~\ref{lem:block-approx-ldp-scale}, it suffices to prove that  there exists $\eta(\varepsilon)>0$ such that for all $\eta \in (0,\eta(\varepsilon))$
\begin{gather}
\limsup_{n\to\infty} \frac{1}{n^2}\log\tilde{\PR}_{n,g_r}(\tbB(\tW,\eta)) \leq - \inf_{f\in \bB(\tW,4\varepsilon) } I_{W_0} (f),
\end{gather}
where $g_r$ is chosen according to Lemma~\ref{lem:block-approx-ldp-scale}.
First, note that
\begin{eq}
\tilde{\PR}_{n,g_r}(\tbB(\tW,\eta)) = \PR_{n,g_r}(\bB(\tW,\eta)).
\end{eq}
Next, we recall a version of Szemer\'edi's regularity lemma from \cite[Theorem 3.1]{Cha17} that will be crucial here (see \cite{szemeredi1978regular} for the original formulation). 
There exists $C(\varepsilon)>0$ and a set  $\sW(\varepsilon)\subset \sW$ with $|\sW(\varepsilon)| \leq C(\varepsilon)$ such that the following holds:
\begin{quote}
    For any $f\in\sW$, there exists $\phi \in \sM$ and $h\in \sW(\varepsilon)$ satisfying $\dcut (f^\phi,h)<\varepsilon$.
\end{quote}
Moreover, for any $h\in \sW(\varepsilon)$, there exists $s\geq 1$ such that $h\in \sW^{\sss (s)}$. 
Without loss of generality, we can additionally assume that the elements of $\sW(\varepsilon)$ are graphons with blocks of equal size. To see this, note that we can approximate each element of $\sW(\varepsilon)$ in $L_2$ by a graphon with equal-sized blocks (see  \cite[Proposition 2.6]{Cha17}).

For empirical graphons corresponding to graphs, the above can be restated as below (see \cite[Theorem 3.1 (iii)]{Cha17}): Recall that $\sM_n$ denotes the set of all permutations of $[n]$, and $G_n^{\sigma_n}$ denotes the graph obtained by relabelling the vertex $i$ by $\sigma_n(i)$, for some $\sigma_n\in \sM_n$. 
Also let us denote $\rB (W,\eta) =  \{W': \dcut(W,W')\leq \eta\}$. 
Then, for any graph $G_n$ on vertex set $[n]$, there exists $\sigma_n\in \sM_n$ and $h\in \sW(\varepsilon)$ such that 
\begin{eq}
W^{\sss G_n^{\sigma_n}} \in \rB(h,\varepsilon).
\end{eq} 
Let $G_n$ be the random graph sampled from the probability distribution $\PR_{n,g_r}$. We define $\rB(\sW(\varepsilon),\varepsilon)=\{g \in \sW: \min_{h \in \sW(\varepsilon)} \dcut(g,h)<\varepsilon\}$, and note that
the above version of the regularity lemma implies that 
\begin{eq}
\{W^{\sss G_n}\in \bB(\tW,\eta) \}&\subseteq \{W^{\sss G_n}\in \bB(\tW,\eta) \} \bigcap \bigg(\bigcup_{\sigma_n\in \sM_n} \{W^{G_n^{\sigma_n}}\in  \rB(\sW(\varepsilon),\varepsilon)\}\bigg) \\
& = \bigcup_{h\in \sW(\varepsilon)}\bigcup_{\sigma_n\in \sM_n}  \{W^{G_n}\in  \bB(\tW,\eta)\}\cap \{W^{G_n^{\sigma_n}}\in \rB(h,\varepsilon)\}.\\
\end{eq}
\noindent Now, $\sW(\varepsilon)$ is a finite set.
Therefore it is enough to show that 
\begin{eq} \label{enough-ub-regularity}
 \limsup_{n\to\infty}&\frac{1}{n^2}\log \PR_{n,g_r}\bigg(\bigcup_{\sigma_n\in \sM_n}\{W^{G_n}\in  \bB(\tW,\eta)\}\cap \{W^{G_n^{\sigma_n}}\in \rB(h,\varepsilon)\}\bigg) \\
 & \qquad \leq - \inf_{f\in \bB(\tW,4\varepsilon) } I_{W_0} (f) ,
\end{eq}
where 
$h\in \sW(\varepsilon)$.
Let $\eta <\varepsilon$.
If the event in \eqref{enough-ub-regularity} is empty, then the bound is trivial. 
In order for the event \eqref{enough-ub-regularity} to be non-empty, we must have that $\delcut (\tW^{G_n},\tW) \leq \eta<\varepsilon$ and $\delcut (\tW^{G_n},\tilde{h})\leq \varepsilon$, so that $\delcut(\tW,\tilde{h}) \leq 2\varepsilon.$
Now, applying Lemma~\ref{lem:relabelling-uniform} yields that the left hand side of \eqref{enough-ub-regularity} is at most
\begin{eq}\label{eq:reduction-ub-1}
 &\limsup_{n\to\infty}\frac{1}{n^2}\log \PR_{n,g_r}\bigg(\bigcup_{\sigma_n\in \sM_n} \{W^{\sss G_n^{\sigma_n}}\in \rB(h,\varepsilon)\}\bigg) \\
 &\leq \limsup_{n\to\infty}\frac{1}{n^2}\max_{\sigma_n\in \sM_n} \log  \PR_{n,g_r}\big( \{W^{\sss G_n^{\sigma_n}}\in \rB(h,\varepsilon)\}\big)\\
 & \leq \limsup_{n\to\infty}\frac{1}{n^2}\max_{\tau \in \T} \log  \PR_{n,g_r}\big( \{W^{\sss G_n,\tau}\in \rB(h,2\varepsilon)\}\big),
\end{eq}where, in the second step, we have also used the fact that $\log n! = o(n^2)$.
Since $\T$ is a finite set, it is now enough to show that for each $\tau\in \T$
\begin{eq}\label{eq:reduction-ub-2}
\limsup_{n\to\infty}\frac{1}{n^2} \log  \PR_{n,g_r}\big( W^{\sss G_n,\tau}\in \rB(h,2\varepsilon)\big)\leq - \inf_{f\in \bB(\tW,4\varepsilon) } I_{W_0} (f).
\end{eq}
Now, by \cite[Lemma 5.4]{Cha17}, $\rB(h,2\varepsilon)$ is closed with respect to the weak topology.
Thus we apply \cite[Theorem 5.1]{Cha17}. 
Although \cite[Theorem 5.1]{Cha17} was stated for the constant graphon, an identical argument could be used to generalize this argument to block constant graphon $g_r$. 
Therefore, \eqref{eq:reduction-ub-2} is at most 
\begin{eq}
 - \inf_{\phi^{-1}\in \sM} \inf_{f\in \rB(h^{\phi},2\varepsilon)} I_{g_r}(f)  \leq -\inf_{f\in \bB(\tilde{h},2\varepsilon)} I_{g_r}(f)  \leq -\inf_{f \in \bB(\tW, 4\varepsilon)} I_{g_r}(f).
\end{eq} 
Now, taking $r\to\infty$, using Lemma~\ref{lem:l1-conv-vont} (note that Lemma~\ref{lem:l1-conv-vont} is stated in terms of $W_0^{n}$, the desired conclusion follows upon substituting $g_r$ in place of $W_0^n$), the proof follows.
  \qed 


\section{Large deviation for uniform graphs with given degree}\label{sec:degrees}
In this section, we complete the proof of Theorem~\ref{thm:ldp-given-degree}.
Using the fact that $(\tsW,\delcut)$ is a compact metric space, it is sufficient 
(see remarks associated to \cite[Theorem 4.5.3]{DZ10}, and \cite[Lemma 4.1]{Cha17}) to show that for any $\tilde{W}\in \tsW$, 
\begin{eq}
\lim_{\eta\to 0}\limsup_{n\to\infty}\frac{1}{n^2} \log\tPRD(\tbB(\tW,\eta)) \leq - J_{D}(\tW)\label{ldp-degree-ub},
\end{eq}and for any $\eta>0$ 
\begin{eq}\liminf_{n\to\infty}\frac{1}{n^2} \log\tPRD(\tbB(\tW,\eta)) \geq - J_{D}(\tW)\label{ldp-degree-lb}.
\end{eq}
\subsection{Key facts from {Chatterjee, Diaconis, Sly \cite{CDS11}}} \label{sec:CDS-facts}
Let us first recall a few key ingredients from \cite{CDS11}, which were used to obtain the graphon limit of $G_{n,\bld{d}}$. 
Let $\hat{\bld{\beta}} = (\hat{\beta}_i)_{i\in [n]}$ be the solution to the system of equations 
\begin{eq}\label{eq:system-of-eqns}
d_i = \sum_{j\neq i} \frac{\e^{\hat{\beta}_i+\hat{\beta}_j}}{1+\e^{ \hat{\beta}_i+\hat{\beta}_j}}, \quad \forall i\in [n].
\end{eq}
Due to \cite[Lemma 4.1]{CDS11}, $\hat{\bld{\beta}}$ exists and $\|\hat{\bld{\beta}}\|_{\infty}\leq C$ for some constant $C>0$ for all sufficiently large $n$ under Assumption~\ref{assmp:degree}. 
It is not obvious that Assumption~\ref{assmp:degree} yields the conditions in \cite[Lemma 4.1]{CDS11}, but that too was shown in the first part of the proof of \cite[Theorem 1.1]{CDS11} in Section 6.2. 
Next, for any $i\neq j$, define 
\begin{eq}\label{beta-model-probab}
\hat{p}_{ij} = \frac{\e^{\hat{\beta}_i+\hat{\beta}_j}}{1+\e^{\hat{\beta}_i+\hat{\beta}_j}},
\end{eq} and let $\hat{G}_n$ be the random graph on vertex set $[n]$ obtained by keeping an edge between vertices $i$ and $j$ with probability $\hat{p}_{ij}$, independently. 
Define 
\begin{eq}\label{defn:W-nd}
W_{n,\bld{d}} (x,y) = 
\begin{cases}
\hat{p}_{ij} \quad &\text{for }x,y \in \big[\frac{i-1}{n},\frac{i}{n}\big)\times \big[\frac{j-1}{n},\frac{j}{n}\big) \text{ and } i\neq j,\\
0 &\text{otherwise}.
\end{cases}
\end{eq}
Since $\|\hat{\bld{\beta}}\|_{\infty} \leq C$, it follows that $(W_{n,\bld{d}})_{n\geq 1}$ is away from the boundary. 
Therefore, the results from Section~\ref{sec:irg} are applicable to $(W_{n,\bld{d}})_{n\geq 1}$.
Next, let $D_n:[0,1]\mapsto [0,1]$ be the step function given by 
\begin{eq}\label{eq:defn-D-n}
D_n(x) = \frac{1}{n}\sum_{j\neq i} \hat{p}_{ij} = \frac{d_i}{n}, \quad \forall x\in \Big[\frac{i-1}{n}, \frac{i}{n}\Big) \text{ and } \forall i \in [n],
\end{eq}
and the degree distribution function is given by 
\begin{eq}
\mu_{D_n}([0,\lambda)) = \Lambda\{x:D_n(x)\leq \lambda \}. 
\end{eq}
By Assumption~\ref{assmp:degree}, $\|D_n-D\|_{\sss L_1} \to 0$, and thus 
\begin{eq}\label{weak-conv-mu}
 \mu_{D_n} \xrightarrow{w} \mu_D,
\end{eq}where $\mu_D$ is defined in \eqref{lim-deg-dist}, where $\xrightarrow{w}$ denotes the weak convergence of measures. 
Define $\beta_n(x) = \sum_{i=1}^n \hat{\beta}_i \mathbbm{1}\{x\in [\frac{i-1}{n}, \frac{i}{n})\}$. 
\citet[page 1430--1432]{CDS11} established that
\begin{eq}\label{eq:graphon-l1-conv}
\|\beta_n - \beta\|_{\sss L_1} \to 0, \quad \|W_{n,\bld{d}}- W_D\|_{\sss L_1} \to 0, 
\end{eq}
 where $\beta$ and $W_D$ are defined in Proposition~\ref{prop:graphon-limit}. 
This fact is critical in our subsequent large deviation analysis.

Next, recall that $\sW_0=\{W\in \sW: \degree_{\tW} = \mu_D\}$ and define $\sW_0^n=\{W\in \sW: \degree_{\tW} = \mu_{D_n}\}$. Note that formally, $\mathrm{deg}_{\tW}$ refers to a cumulative distribution function, and not to the associated probability measure. We use these notions interchangeably, and not overload the notation henceforth. 
Given any graphon $W_0^n \in \sW^{\sss (n)}_{\sss \mathrm{IRG}}$, recall the definition of the probability measure $\PR_{n,W_0^n}$ from Section~\ref{sec:irg}.
Note that, under  $\PR_{n,W_{n,\bld{d}}}$ with $W_{n,\bld{d}}$ given by  \eqref{defn:W-nd}, the probability of producing a particular graph with degree sequence $\bld{d}$ is given by $\e^{\sum_{i\in [n]}\hat{\beta}_{i}d_i}/ \prod_{i<j} (1+\e^{\hat{\beta}_i+\hat{\beta}_j})$. 
Therefore, the conditional law of $\PR_{n,W_{n,\bld{d}}}$, conditionally on degree sequence $\bld{d}$, is uniform among all the graphs with degree sequence $\bld{d}$. More formally, 
\begin{eq}\label{eq:dist-comparison}
\PR_{n,W_{n,\bld{d}}}(\cdot \vert \sW_0^n) = \PR_{n,\bld{d}} (\cdot). 
\end{eq}
Next we quote a key lemma from \cite{CDS11} which will be used in the proof: 
Let $(r_{ij})_{i\neq j}$ satisfy $r_{ij} = r_{ji}$, $r_{ii} = 0$ and $\sum_{j\in [n]\setminus\{i\}}r_{ij} = d_i$, and construct a random graph $G_n$ on the vertex set $[n]$ by keeping an edge between $i$ and $j$ with probability $r_{ij}$.
\begin{lemma}[{\cite[Lemma 6.2]{CDS11}}] \label{lem:CDS-counting}
For all sufficiently large $n$, $G_n$ has degree sequence exactly $\bld{d}$ with probability at least $\e^{-n^{7/4}}$.
\end{lemma}
\noindent A direct corollary of Lemma~\ref{lem:CDS-counting} is the following: 
\begin{eq}\label{eq:prob-simple-lb}
\PR_{n,W_{n,\bld{d}}}(\sW_0^n) \geq \e^{- n^{7/4}},
\end{eq}for all sufficiently large $n$. 
We are now ready to prove our $\ldp$ result.

\subsection{Proof of the upper bound {\eqref{ldp-degree-ub}}}\label{sec:ub}
Define the L\'evy-Prokhorov distance \cite{Pro56} between two distribution functions $F_1,F_2$ supported on $[0,1]$ by 
\begin{eq}
d_{\sss \mathrm{LP}} (F_1,F_2) = \inf \big\{\varepsilon>0: F_2(\lambda-\varepsilon) - \varepsilon \leq F_1(\lambda)\leq F_2(\lambda+\varepsilon)+\varepsilon, \ \forall\, \lambda\, \in [0,1]\big\}. 
\end{eq}
This distance can be naturally defined for any two probability measures supported on [0,1] (via their distribution functions), and induces a metric on this space. In fact, $d_{\sss \mathrm{LP}}$ metrizes the weak convergence of probability measures on [0,1] (see \cite{Pro56}). 
Using \cite[Theorem 2.16]{BCCG15}, it follows that 
\begin{eq}\label{bound-prokhorov-metric}
d_{\sss \mathrm{LP}}(\degree_{\tW_1},\degree_{\tW_2}) \leq \big(2\delcut(\tW_1,\tW_2))^{1/2}.
\end{eq}
To prove \eqref{ldp-degree-ub}, we will be assuming that $\tW\in \tsW_0$. If that is not the case, then the logarithm of probability in \eqref{ldp-degree-ub} is $-\infty$ for all sufficiently large $n$ and small $\eta$. 
To see this, suppose $\tW\in \tsW$ is such that $d_{\sss \mathrm{LP}} (\degree_{\tW},\mu_D) = c >0$. 
 Since $d_{\sss \mathrm{LP}} (\mu_{D_n},\mu_D)\to 0$ by \eqref{weak-conv-mu}, it follows that,  for all sufficiently large $n$,
 $d_{\sss\mathrm{LP}}(\degree_{\tW},\mu_{D_n}) \geq c/2$. 
Take $\eta_0 = c^2/32$.  
Now, for any $\tU\in \tbB(\tW,\eta_0)$
\begin{eq}
 \frac{c}{2} \leq d_{\sss \mathrm{LP}}(\degree_{\tW},\mu_{D_n}) \leq d_{\sss \mathrm{LP}}(\degree_{\tW},\degree_{\tU})+ d_{\sss \mathrm{LP}}(\degree_{\tU},\mu_{D_n}) \leq \frac{c}{4} + d_{\sss \mathrm{LP}}(\degree_{\tU},\mu_{D_n}), 
\end{eq}
 for all sufficiently large $n$, where the final step follows from \eqref{bound-prokhorov-metric}.
 Thus $d_{\sss \mathrm{LP}}(\degree_{\tU},\mu_{D_n})\geq c/4$ for all $\tU\in \tbB(\tW,\eta_0)$, and thus $\tPRD (\tbB(\tW,\eta_0)) = 0$.

Therefore, we will assume that $\degree_{\tW} = \mu_D$.
Note that, 
\begin{eq}\label{eq:ub-simple-1}
\tPRD(\tbB(\tW,\eta)) &= \PRD(\bB(\tW,\eta)) = \frac{\PR_{n,W_{n,\bld{d}}}(\bB(\tW,\eta)\cap \sW_0^n) }{\PR_{n,W_{n,\bld{d}}}(\sW_0^n) } \\
&\leq \e^{n^{7/4}} \PR_{n,W_{n,\bld{d}}}(\bB(\tW,\eta)),
\end{eq}where the second equality follows from \eqref{eq:dist-comparison} and the last step follows from \eqref{eq:prob-simple-lb}.
Now, using Proposition~\ref{thm:ldp-IRG}, for any $\varepsilon>0$, there exists $\eta(\varepsilon)>0$ such that for all $\eta\in (0,\eta(\varepsilon))$,  we have
\begin{eq}
\limsup_{n\to\infty}\frac{1}{n^2} \log \tPRD(\tbB(\tW,\eta)) \leq -\inf_{f\in \bB(\tW,4\varepsilon)} I_{W_D}(f)+\varepsilon.
\end{eq}
Consequently, \eqref{ldp-degree-ub} follows upon sending $\eta \to 0$ and then $\varepsilon\to 0$. 

%

%

\subsection{Proof of the lower bound~\eqref{ldp-degree-lb}} \label{sec:ldp-lower-bound}
Fix $\tW\in \tsW$ such that  $\degree_{\tW} = \mu_D$, otherwise the rate function is $-\infty$, and the lower bound is trivial.
Recall that $\sW_0=\{W\in \sW: \degree_{\tW} = \mu_D\}$, and the definition of $\hat{G}_n$ from Section~\ref{sec:CDS-facts}. Define the event $$\cE_n = \{\exists g\in \sW_0\text{ with } \delcut(\tg,\tW) \leq \eta \text{ such that } \delcut(\tW^{\sss G_{n,\bld{d}}}, \tg)\leq \eta\}.$$
Note that, if $\cE_n$ happens, then $\delcut(\tW^{\sss G_{n,\bld{d}}},\tg)\leq \eta$, and therefore, by the triangle inequality, $\delcut(\tW^{\sss G_{n,\bld{d}}},\tW)\leq 2\eta$. 
Next, note that for any collection of events $(A_\alpha)_{\alpha\in \cA}$, $\PR(\cup_{\alpha\in \cA }A_{\alpha}) \geq \max_{\alpha \in \cA} \PR(A_{\alpha})$.
Thus, we have
\begin{eq}\label{lower-bound-reduction-sup}
\tPRD(\tbB(\tW,2\eta)) \geq \tPRD(\cE_n) \geq  \sup_{g\in \bB (\tW,\eta) \cap \sW_0} \tPRD (\tbB(\tilde{g},\eta)).
\end{eq}
The lower bound on \eqref{lower-bound-reduction-sup} may seem artificial at first, but its technical  significance will become clear later in \eqref{eq:lb-wo-envelope}, \eqref{eq:lb-with-envelope}, while proving  the LDP lower bound in terms of the lower semi-continuous envelope $J_{W_D}$.
Our focus will be to lower bound $\tPRD (\tbB(\tilde{g},\eta))$.
The following lemma is a crucial ingredient which states that graphons with any fixed degree distribution function can be approximated by piecewise constant graphons with approximately the same degree function. 
We first state this lemma and complete the proof of the lower bound. The proof of the lemma is given at the end of this section.
Recall the definition of $\sW_{\sss \mathrm{IRG}}^{\sss (n)}$ from Section~\ref{sec:irg}.
For $h_n\in \sW_{\sss \mathrm{IRG}}^{\sss (n)}$, let $(h_{ij})_{i,j\in [n]}$ be the values of $h_n$ on the blocks $S_{ij}$, where $S_{ij}:=[\frac{i-1}{n},\frac{i}{n}) \times[\frac{j-1}{n},\frac{j}{n})$.
%
For any $\sigma_n \in \sM_n$, we define the graphon $h_n^{\sigma_n}$ by $h_n^{\sigma_n} (x,y) = h_{\sigma_n (i) \sigma_n (j)}$ for all $x,y\in S_{ij}$, $i,j\in [n]$.

\begin{lemma}\label{lem:construction-step-function}
Let $g\in \sW_0$, i.e., $\deg_{\tilde{g}} = \mu_D$.  
Further, let $D_n$ be a step function of the form \eqref{eq:defn-D-n} such that $\|D_n-D\|_{\sss L_1} \to 0$.
There exist graphons $(g_n)_{n\geq 1}$ and $(\sigma_{0n})_{n\geq 1}$ with $\sigma_{0n}\in \sM_n$ such that $\|g_n^{\sigma_{0n}}-g\|_{\sss L_1} \to 0$, and there exists an $n_0$ (independent of $g$) such that for all $n\geq n_0$, we have  $\int_0^1 g_n(x,y) \dif y= D_n(x)$, and 
\begin{eq} \label{step-approx-graphon}
g_{n} (x,y) = 
\begin{cases}
g_{ij}, \quad &x,y \in \big[\frac{i-1}{n},\frac{i}{n}\big)\times \big[\frac{j-1}{n},\frac{j}{n}\big), \quad i\neq j\\
0 &\text{otherwise}.
\end{cases}
\end{eq}
where $n^{-1}<g_{ij}<1-n^{-1}$.
\end{lemma}

Next, since $\|D_n-D\|_{\sss L_1} \to 0$ by Assumption~\ref{assmp:degree}, using Lemma~\ref{lem:construction-step-function}, we can construct a function $g_n$ with $\delcut(\tg_n,\tg) \to 0$ such that \eqref{step-approx-graphon} holds,
and $\sum_{j \in [n] \setminus \{i\}} g_{ij} = d_i$ for all $i\in [n]$.
Also let $G_n$ denote the graph on vertex set $[n]$, where an edge between vertices $i$ and $j$ are kept with probability $ h_{ij} =g_{\sigma_{0n}(i) \sigma_{0n}(j)}$, independently, where $\sigma_{0n} \in \sM_n$ is given by Lemma~\ref{lem:construction-step-function}. Let $\PR_{n,h_n}$ denote the distribution of $W^{\sss G_n}$. 
By our construction in Lemma~\ref{lem:construction-step-function}, we have that $\|h_n - g\|_{\sss L_1} \to 0$.
Using \eqref{eq:dist-comparison}, we can write
\begin{eq}\label{split-up-g-n}
&\tPRD(\tbB(\tilde{g},\eta ))  = \frac{\PR_{\sss n, W_{n,\bld{d}}}(\bB(\tilde{g},\eta )\cap \sW_0^n)}{\PR_{\sss n, W_{n,\bld{d}}}(\sW_0^n)}\geq  \int_{\bB(\tilde{g},\eta)\cap \sW_0^n} \dif \PR_{n,W_{n,\bld{d}}} \\
&=\int_{\bB(\tilde{g},\eta ) \cap \sW_0^n} \e^{-\log \frac{\dif \PR_{n,h_n}}{\dif \PR_{\sss n,W_{n,\bld{d}}}}} \dif \PR_{n,h_n} \\
& =  \PR_{n,h_n}(\bB(\tilde{g},\eta )  \cap \sW_0^n) \frac{1}{\PR_{n,h_n}(\bB(\tilde{g},\eta ) \cap \sW_0^n)} \int_{\bB(\tilde{g},\eta ) \cap \sW_0^n} \e^{-\log \frac{\dif \PR_{n,h_n}}{\dif \PR_{n,W_{n,\bld{d}}}}} \dif \PR_{n,h_n}.
\end{eq}
Now, taking logarithms and using Jensen's inequality, the above is at least 
\begin{eq}\label{eq:simple-jensen}
& \log \PR_{n,h_n}(\bB(\tilde{g},\eta ) \cap \sW_0^n) - \frac{1}{\PR_{n,h_n}(\bB(\tilde{g},\eta ) \cap \sW_0^n)} \int_{\bB(\tilde{g},\eta ) \cap \sW_0^n} \log \frac{\dif \PR_{n,h_n}}{\dif \PR_{n,W_{n,\bld{d}}}} \dif \PR_{n,h_n}.
\end{eq}
Denote the two terms above by $(I)$ and $(II)$ respectively. 
To deal with the term $(I)$, we need the following lemma:
\begin{lemma}\label{lem:conditional-convergence}
For any $\eta>0$, as $n\to\infty$, 
\begin{eq} \label{eq:cond-conv-irg}
\PR_{n,h_n} (\bB(\tilde{g},\eta) \vert \sW_0^n) \to 1.
\end{eq}
\end{lemma}
\begin{proof}
We denote the random graph sampled according to probability measures $\PR_{n,h_n}(\cdot)$ by~$G_n$,  $h_n = g_n^{\sigma_{0n}}$, and recall the definition of subgraph densities from Definition~\ref{defn:subgraph-dinsity}. 
Since $\delcut(\tg_n,\tg)\to 0$, it follows using \cite[Lemma 10.23]{L2012} that $t(F,g_n) \to t(F,g)$ for any finite simple graph $F$.
It is enough to show that, $t(F,W^{G_n}) \to t(F,g)$ almost surely with respect to the measure $\bigotimes_{n\geq 1}\PR_{n,h_n} (\cdot \vert \sW_0^n)$ for any fixed finite simple graph $F$, since then the proof will follow using \cite[Lemma 10.32]{L2012}.

First,  $\E_{n,h_n}[t(F,W^{\sss G_n})]  = t(F,g_n)\to t(F,g)$, and a standard argument  using the bounded difference inequality (cf.~\cite[Lemma 6.1]{CDS11})
yields for any $\varepsilon >0$
\begin{eq} \label{eq:density-concentration}
\PR_{n,h_n} (|t(F,W^{\sss G_n}) - \E_{n,h_n}[t(F,W^{\sss G_n})]|>\varepsilon) \leq 2\e^{-C\varepsilon^2n^2},
\end{eq}
for some constant $C>0$.
Now, recall that $\sum_{j\neq i}g_{ij} = d_i$ by construction. 
We aim to apply Lemma~\ref{lem:CDS-counting}. 
$h_{ij}$ is obtained from $g_{ij}$ by vertex relabelling, and thus Lemma~\ref{lem:CDS-counting} is also applicable to $G_n$.
Thus, it follows that
\begin{eq}
&\PR_{n,h_n} (|t(F,W^{\sss G_n}) - \E_{n,h_n}[t(F,W^{\sss G_n})]|>\varepsilon\vert \sW_0^n) \\
&\quad = \frac{\PR_{n,h_n} (\{|t(F,W^{\sss G_n}) - \E_{n,h_n}[t(F,W^{\sss G_n})]|>\varepsilon\}\cap \sW_0^n)}{\PR_{n,h_n}(\sW_0^n)} \\
&\quad \leq \frac{\PR_{n,h_n} (|t(F,W^{\sss G_n}) - \E_{n,h_n}[t(F,W^{\sss G_n})]|>\varepsilon)}{\PR_{n,h_n}(\sW_0^n)} \leq 2\e^{-Cn^2},
\end{eq}for some constant $C>0$. Now the required almost sure convergence follows using the Borel–Cantelli lemma.
This completes the proof.
\end{proof}
\noindent {\it Completing the proof of the lower bound.}
Note that, by Lemmas~\ref{lem:CDS-counting} and \ref{lem:conditional-convergence}, the term $(I)$ in \eqref{eq:simple-jensen} simplifies to
\begin{eq}\label{eq:simple-I}
(I) =\log \PR_{n,h_n}(\bB(\tilde{g},\eta ) \cap \sW_0^n)= \log \PR_{n,h_n}( \sW_0^n) + o(1) \geq - C n^{7/4} = o(n^2), 
\end{eq}for some constant $C>0$.
To analyze term $(II)$, firstly note that 
\begin{eq}\label{eq:lb-term-2-exact}
\log \frac{\dif \PR_{n,h_n}}{\dif \PR_{n,W_{n,\bld{d}}}}= \sum_{1\leq i<j\leq n} \bigg(I_{ij} \log\Big(\frac{h_{ij}}{\hat{p}_{ij}}\Big)+(1-I_{ij}) \log\Big(\frac{1-h_{ij}}{1-\hat{p}_{ij}}\Big)\bigg),
\end{eq}where $I_{ij} \sim \mathrm{Ber}(h_{ij})$ independently, and $\hat{p}_{ij}$ is defined in \eqref{beta-model-probab}. 
By changing one $I_{ij}$, this quantity can change by at most \begin{eq}\label{sup-entropy}
\max_{i,j} \Big| \log\Big(\frac{h_{ij}}{\hat{p}_{ij}}\Big) \Big| + \Big| \log\Big(\frac{1-h_{ij}}{1-\hat{p}_{ij}}\Big)\Big|\leq C \log n,
\end{eq} using the condition from Lemma~\ref{lem:construction-step-function} that $n^{-1}<g_{ij}<1-n^{-1}$, and $W_D$ (and thus also $(W_{n,\bld{d}})_{n\geq 1}$)  is away from the boundary.
Therefore, an application of Azuma-Hoeffding inequality \cite[Theorem 2.8]{boucheron2013concentration} yields  
\begin{eq}\label{eq:bad-event}
\PR_{n,h_n} \bigg(\bigg|\log \frac{\dif \PR_{n,h_n}}{\dif \PR_{n,W_{n,\bld{d}}}} - \E_{n,h_n}\Big[\log \frac{\dif \PR_{n,h_n}}{\dif\PR_{n,W_{n,\bld{d}}}}\Big]\bigg|>\varepsilon_n  n^2\log n\bigg) \leq 2\e^{-C' \varepsilon_n^2 n^2},
\end{eq}for some constant $C'>0$ which depends on the constant in \eqref{sup-entropy}.
We denote the event in~\eqref{eq:bad-event} by $\mathcal{A}_n$. 
Take $\varepsilon_n  = n^{-1/10}$.
Note that, on $\mathcal{A}_n^c$, 
\begin{eq}
\log \frac{\dif \PR_{n,h_n}}{\dif \PR_{n,W_{n,\bld{d}}}} 
& \leq \E_{n,g_n}\Big[\log \frac{\dif \PR_{n,h_n}}{\dif \PR_{n,W_{n,\bld{d}}}}\Big]+ n^{19/10}.
\end{eq}
Also note that, by \eqref{sup-entropy}, the log derivative $\log \frac{\dif \PR_{n,h_n}}{\dif \PR_{n,W_{n,\bld{d}}}}$ is at most $Cn^2 \log n.$
Therefore, 
\begin{eq}\label{eq:II-1}
(II) &\leq \frac{1}{\PR_{n,h_n}(\bB(\tilde{g},\eta ) \cap \sW_0^n)} \int_{\sW_0^n} \log \frac{\dif \PR_{n,h_n}}{\dif \PR_{n,W_{n,\bld{d}}}} \dif \PR_{n,h_n} \\
& \leq \frac{1}{\PR_{n,h_n}(\bB(\tilde{g},\eta ) \cap \sW_0^n)}\bigg(\E_{n,h_n}\Big[\log \frac{\dif \PR_{n,h_n}}{\dif \PR_{n,W_{n,\bld{d}}}}\Big]+ n^{19/10}\bigg) \PR_{n,h_n}(\sW_0^n) \\
&\qquad + \frac{(2Cn^2 \log n)\e^{-C'n^{9/5}}}{\PR_{n,h_n}(\bB(\tilde{g},\eta ) \cap \sW_0^n)}\\
&= \frac{1}{\PR_{n,h_n}(\bB(\tilde{g},\eta ) \vert \sW_0^n)}\bigg(\E_{n,h_n}\Big[\log \frac{\dif \PR_{n,h_n}}{\dif \PR_{n,W_{n,\bld{d}}}}\Big]+ o(n^{2})\bigg)  +o(n^2) \\ 
&= (1+o(1)) \E_{n,h_n}\Big[\log \frac{\dif \PR_{n,h_n}}{\dif \PR_{n,W_{n,\bld{d}}}}\Big]+ o(n^{2}),
\end{eq}
where the last-but-one step follows applying  \eqref{eq:simple-I}, and the last step use Lemma~\ref{lem:conditional-convergence}. 
Further, since $\|h_n - g\|_{\sss L_1} \to 0$, we also have $\|h_n - g\|_{\sss L_2} \to 0$ since $h_n$, $g$ takes values in a bounded interval $[0,1]$, and consequently,  \eqref{eq:lb-term-2-exact} yields that 
\begin{eq}\label{eq:II-2}
\lim_{n\to\infty}\frac{1}{n^2}\E_{n,h_n}\Big[\log \frac{\dif \PR_{n,h_n}}{\dif \PR_{n,W_{n,\bld{d}}}}\Big] = I_{W_D}(g).
\end{eq}
See also \cite[Lemma 5.7]{Cha17} for more details for proving an analogue of \eqref{eq:II-2} with $W_D = p$.
The argument here is identical.
Thus, combining  \eqref{eq:simple-I}, \eqref{eq:II-1} and \eqref{eq:II-2}, we have 
\begin{eq}\label{eq:lb-wo-envelope}
\liminf_{n\to\infty} \frac{1}{n^2} \log \tPRD(\tbB(\tilde{g},\eta )) \geq - I_{W_D}(g).
\end{eq}
Thus, 
 \eqref{lower-bound-reduction-sup} yields that
 \begin{eq}\label{eq:lb-with-envelope}
 &\liminf_{n\to\infty}\frac{1}{n^2} \log \tPRD(\tbB(\tW,2\eta )) \blue{\geq} \liminf_{n\to\infty} \sup_{g\in \bB (\tW,\eta) \cap \sW_0} \frac{1}{n^2} \log \tPRD (\tbB(\tilde{g},\eta))  \\
 &\geq  \sup_{g\in \bB (\tW,\eta) \cap \sW_0} \liminf_{n\to\infty} \frac{1}{n^2} \log \tPRD (\tbB(\tilde{g},\eta))\geq - \inf_{g\in \bB(\tilde{W},\eta) \cap \sW_0} I_{W_D}(g), 
 \end{eq}
  which concludes the proof of the lower bound in \eqref{ldp-degree-lb}.
\qed 
\\

It remains to prove Lemma~\ref{lem:construction-step-function}. 
To this end, we will need the following ingredient: For any Borel measurable function $g:[0,1] \mapsto [0,1]$, let $m_g(z) = \Lambda(\{y:g(y) > z\})$, where we recall that $\Lambda$ is the Lebesgue measure. 
\begin{lemma}\label{lem:monotone-function-L1}
Let $(f_n)_{n\geq 1}$ and $f$ be such that $f_n,f: [0,1]\mapsto [0,1]$ are non-increasing, Borel measurable functions. 
Suppose that $\lim_{n\to\infty}m_{f_n}(z) = m_f(z)$ for all continuity points $z$ of $m_f$.
Then, as $n\to\infty$, $\|f_n-f\|_{\sss L_1} \to 0$.
\end{lemma}
\begin{proof}
For any Borel measurable function $g:[0,1] \mapsto [0,1]$, the monotone rearrangement is defined as $g^*(x) = \inf\{z: m_g(z)\leq x\}$.
We will prove the following two facts about the monotone rearrangement:
\begin{fact}\label{fact:mon-re-dec}
If $f$ is non-increasing, then $f^* = f$ almost surely.
\end{fact}
\begin{fact}\label{fact:conv-mon-re}
If $m_{f_n}(z) \to m_f(z)$ for all continuity points $z$ of $m_f$, then $f_n^* \to f^*$ almost surely, as $n\to\infty$.
\end{fact}
Using Facts~\ref{fact:mon-re-dec},~and~\ref{fact:conv-mon-re}, it follows that $f_n \to f$ almost surely. 
Thus the proof follows by the dominated convergence theorem. 
\end{proof}
\begin{proof}[Proof of Fact~\ref{fact:mon-re-dec}]
Since $f$ is non-increasing, we have that $\{y:f(y)>f(x)\} \subset \{y:y\leq x\}$.
This implies $m_f(f(x)) = \Lambda(\{y:f(y)>f(x)\}) \leq x$, and thus $f^*(x) = \inf \{z: m_f(z) \leq x\} \leq f(x)$. 
Now, let $x$ be a continuity point of $f$, and fix $\varepsilon>0$. 
Then, $m_f(f(x) -\varepsilon) = \Lambda (\{y:f(y)>f(x) -\varepsilon\})>x$.
Now, since $m_f$ is non-increasing, whenever $m_f(z) \leq x$, we have $z>f(x) - \varepsilon$. This implies that $f^*(x) \geq f(x) -\varepsilon$, and thus $f^*(x) =f(x)$ whenever $x$ is a continuity point of $f$. 
Now, the proof follows using the fact that any non-increasing function can only have countably many points of discontinuity.
\end{proof}
\begin{proof}[Proof of Fact~\ref{fact:conv-mon-re}]
First note that, whenever $z_n\searrow z$, we have  $m_{f}(z_n) \nearrow m_f(z)$, and thus $m_f$ is right-continuous. 
Next, for any $z< f^*(x)$, we have $m_f(z)>x$. 
Let $z$ be a continuity point of $m_f$.
Since $m_{f_n}(z) \to m_f(z)$, for all sufficiently large $n$, we have $m_{f_n} (z)>x$, and thus $z\leq \liminf_{n\to\infty}f_n^*(x)$. 
Therefore, $\liminf_{n\to\infty}f_n^*(x) \geq f^*(x)$.

Next, let $x$ be a continuity point of $f^*$, i.e., for all $\varepsilon>0$, there exists a $\delta>0$ such that $f^*(x-\delta) <f^*(x)+\varepsilon$.
Define $\xi = \limsup_{n\to\infty} f_n^*(x)$. 
Then, there exists $(n_k)_{k\geq 1}\subset \N$ such that for all $k\geq 1$, $f_{n_k}^*(x)> \xi -\varepsilon$, and thus $m_{f_{n_k}}(\xi -\varepsilon)>x$.
Now, since $m_f$ has countably many points of discontinuity, we can choose $\varepsilon>0$ such that $\xi -\varepsilon$ is a continuity point of $m_f$. 
This implies that $m_f(\xi -\varepsilon) \geq x>x-\delta$, and thus $f^*(x-\delta) > \xi - \varepsilon$. 
Thus, $f^*(x) > \xi - 2\varepsilon = \limsup_{n\to\infty} f_n^*(x) -2\varepsilon$.
The proof again follows using the fact that $f^*$ can have only countably many points of discontinuity.
\end{proof}

\begin{proof}[Proof of Lemma~\ref{lem:construction-step-function}]
Recall that $\sW^{\sss (n)}_{\sss \mathrm{IRG}}$ denotes the collection of piecewise constant graphons defined below \eqref{eq:piecewise-const}, and also that $S_{ij}:=[\frac{i-1}{n},\frac{i}{n}) \times[\frac{j-1}{n},\frac{j}{n})$. For $W\in \sW^{\sss (n)}_{\sss \mathrm{IRG}}$, we write $W_{ij}$ to denote the value of $W$ on $S_{ij}$. 
In order to produce a block-constant graphon that is close to $g$ and degree function exactly equal to $D_n$, we proceed via following steps.
\\

\noindent \textbf{Step 1: Approximation by block constant graphons.}
For $1\leq i\neq j\leq n$, define 
\begin{eq}\label{step}
g_{n1}(x,y) = n^2 \int_{S_{ij}} g(u,v)\dif u\dif v, \quad \forall (x,y) \in S_{ij},
\end{eq}
and $g_{n1}(x,y) = 0$ otherwise.
A standard argument implies that $\|g_{n1}-g\|_{\sss L_2} \to 0$ (see  \cite[Proposition 2.6]{Cha17}).
Since $g$ is bounded, we also have $\|g_{n1}-g\|_{\sss L_1} \to 0$. \\

\noindent \textbf{Step 2: $L_1$-approximation of the degree function.}
Let $\sigma_n$ be  any permutation such that $(\sum_{j\in [n]}g_{n1, \sigma_n(i)\sigma_n(j)})_{i\in [n]}$ is non-increasing, and let $g_{n2}$ take value $g_{n1,\sigma_n(i)\sigma_n(j)}$ on $S_{ij}$. 
Since $\|g_{n1}-g\|_{\sss L_1} \to 0$, it follows that $ \|g_{n2}^{\sigma_{0n}}- g\|_{\sss L_1} \to 0$, where $\sigma_{0n}$ is the inverse permutation of~$\sigma_n$. 
Let $f_{n2} (x)= \int_0^1 g_{n2} (x,y) \dif y$, which is now non-increasing by our construction.
Using~\eqref{bound-prokhorov-metric}, we can now apply Lemma~\ref{lem:monotone-function-L1} with $(f_{n2})_{n\geq 1}$ and $D$.
Thus, we have $\|f_{n2}-D\|_{\sss L_1} \to 0$.
Recall that $\|D_n-D\|_{\sss L_1} \to 0$ by Assumption~\ref{assmp:degree}, and thus it follows that $\|f_{n2}-D_n\|_{\sss L_1} \to 0$. \\

\noindent \textbf{Step 3: $L_\infty$-approximation of the degree function.}
By Markov's inequality, there exists $(\varepsilon_n)_{n\geq 1}$ with $\varepsilon_n\to 0$ such that 
\begin{eq}
\frac{|V_{\sss \mathrm{ex}}|}{n} \to 0,\quad \text{where } V_{\sss \mathrm{ex}}:=\bigg\{i: \bigg|\sum_{j\in [n]} g_{n2,ij}-d_i\bigg|>n\varepsilon_n\bigg\}. 
\end{eq}
Let $g_{n3}(x,y) = \hat{p}_{ij}$ if $(x,y)\in S_{ij}$ with  $i\in V_{\sss\mathrm{ex}}$ or  $j\in V_{\sss\mathrm{ex}}$, and $g_{n3}(x,y) = g_{n2}(x,y)$ otherwise. 
This changes at most $2n|V_{\sss \mathrm{ex}}|$ block values of $g_{2n}$, and therefore, $\|g_{n3} - g_{n2}\|_{\sss L_1} \leq 2|V_{\sss \mathrm{ex}}|/n \to 0$.
To campare the degree functions, note that for $i\in V_{\sss \mathrm{ex}}$, we have $\sum_{j\in [n]} g_{n3,ij}=d_i$, and for $i\notin V_{\sss \mathrm{ex}}$, we have $|\sum_{j\in [n]} g_{n3,ij} - d_i| \leq n \varepsilon_n + |V_{\sss \mathrm{ex}}|$. 
Thus, if $f_{n3} (x)= \int_0^1 g_{n3} (x,y) \dif y$, then $\|f_{n3} - D_n\|_{\sss L_{\infty}} \to 0$ by our construction. \\

\noindent \textbf{Step 4: Truncation away from 0,1.} 
Let $\delta_n:= \max\{\|f_{n3} - D_n\|_{\sss L_{\infty}} , n^{-1/2}\}$. 
Define  on $S_{ij}$ for $i\neq j$,
\begin{eq}
g_{n4} (x,y):= 
\begin{cases}
 g_{n3}(x,y) & \quad\text{ if }\quad  \delta_n \leq g_{n3}(x,y) \leq 1- \delta_n \\
\frac{1}{n}  & \quad \text{ if }\quad  g_{n3}(x,y)<\delta_n  \\
  1-\frac{1}{n}  & \quad \text{ if }\quad g_{n3}(x,y)>1-\delta_n .
\end{cases} 
\end{eq}
and $g_{n4}(x,y) = 0$ on $S_{ii}$ for all $i$.
By construction, $\|g_{n4} - g_{n3}\|_{\sss L_{\infty}} \leq \delta_n $ and hence $\|f_{n4} - D_n\|_{\sss L_{\infty}} \leq \|f_{n4} - f_{n3}\|_{\sss L_{\infty}} + \|f_{n3} - D_n\|_{\sss L_{\infty}} \leq 2\delta_n$, where $f_{n4} (x)= \int_0^1 g_{n4} (x,y) \dif y$. \\

\noindent \textbf{Step 5: Producing a graphon with exact degree function $D_n$.}
We need the following: 
\begin{fact}\label{fact:solution-fixed-sum}
Given any sequence $a= (a_i)_{i\in [n]}$, it is possible to find weights  $w= (w_{ij})_{i,j}$ with $w_{ij}=w_{ji}$ for all $i,j$, such that $\sum_{j\in [n]\setminus \{i\}}w_{ij} = a_i$, and $w_{ii} = 0$ for all $i\in [n]$, and 
\begin{eq}\label{linfty-w-ij}
\|w\|_{\infty} \leq  \frac{\|a\|_{\infty} }{n-2}+\frac{\|a\|_1}{(n-1)(n-2)}.
\end{eq} 
\end{fact}
Let us first complete the proof of   
Lemma~\ref{lem:construction-step-function}; the proof of Fact~\ref{fact:solution-fixed-sum} is given subsequently. 
Note that $D_n-f_{n4}$ is a step function with constant values in $[\frac{i-1}{n},\frac{i}{n})$ for all $i$.  
We take $a_i$ to be the value of $n(D_n-f_{n4})$ on $[\frac{i-1}{n},\frac{i}{n})$. 
Now, we choose $w$ according to Fact~\ref{fact:solution-fixed-sum}, and define 
\begin{eq}
g_{n5} (x,y) =
\begin{cases}
 g_{n4} (x,y) + w_{ij} \quad &\forall (x,y) \in S_{ij}, i\neq j,\\
\quad 0  &\text{otherwise.}
\end{cases}
\end{eq}
Recall the bounds from Step 4. Since $\|a\|_{\infty} \leq n\delta_n$ and $\|a\|_{1} \leq n \|a\|_{\infty} \leq n^2 \delta_n$, we have from \eqref{linfty-w-ij} that $\|w\|_{\infty} \leq 3\delta_n$.
Thus, $\|g_{n5} -g_{n4}\|_{\sss L_\infty} \to 0$, and moreover  $\int_0^1 g_{n5} (x,y)\dif y = D_n(x)$. However, $g_{n5}$ can take values in $[-3\delta_n,1+3\delta_n]$ and $\delta_n \leq n^{-1/2}$. 
Finally we define $g_n = n^{-1/4} W_{\sss n,\bld{d}}+(1-n^{-1/4}) g_{n5}$. Since $W_{\sss n,\bld{d}}$ is away from boundary, it follows that $g_{n}$ takes values in $[\frac{1}{n},1-\frac{1}{n}]$ for all sufficiently large $n$, and also $\int g_{n} (x,y)\dif y = D_n(x)$. This completes the proof of Lemma~\ref{lem:construction-step-function}. 
\end{proof}

\begin{proof}[Proof of Fact~\ref{fact:solution-fixed-sum}]
Let us view $w$ as a vector with its elements indexed by $(j,k)$, $j<k$. We wish to find a solution of $w$ in the equation $Mw = a$, where $M$ is an $n \times {n \choose 2}$ matrix with entries $m_{i, (j,k)} = \ind{i\in \{j,k\}}$.
First let us find the inverse of $MM^T$. Indeed, 
\begin{eq}
(MM^T)_{uv} = \sum_{j<k}\ind{u\in \{j,k\}}\ind{v\in \{j,k\}} = \begin{cases}
 1 &\text{ if } u\neq v,\\
 n-1 &\text{ if }u=v.
\end{cases}
\end{eq}
Thus $MM^T = (n-2)I + 1 1^T$. An application of Sherman-Morrison formula (see e.g. \cite{hager1989updating}) yields that 
\begin{eq}
(MM^T)^{-1} = \frac{I}{n-2} - \frac{11^T}{2(n-1)(n-2)}. 
\end{eq}
Now, $w = M^T (MM^T)^{-1} a$ is a solution to the equation $Mw = a$. 
Also, the $(j,k)$-th column of $M$ consists of 1 on the $j$-th and $k$-th entries and zero elsewhere. 
Hence, we observe that $\|w\|_{\infty} \leq 2\|a\|_{\infty}/(n-2) +  \|a\|_1/(n-1)(n-2)$, and the proof follows.
\end{proof}

\section{Proofs of Corollaries~\ref{cor:ldp-cont-func} and \ref{cor:partition-function}} \label{sec:proof-cor}
\subsection{Large deviation for continuous functionals} \label{sec:ldp-cont-func}
In this section, we prove Corollary~\ref{cor:ldp-cont-func}, leveraging the general techniques used in \cite[Section 3]{CV11} and \cite[Section 3.2]{DL18}. 
\begin{proof}[Proof of Corollary~\ref{cor:ldp-cont-func}~(1)]
Let $\Gamma_{\geq r} = \{\tW : \tau (\tW) \geq r\}$. This is a closed set, since $\tau$ is continuous.  Recall that  $\sW_0=\{W\in \sW: \degree_{\tW} = \mu_D\}$ and $\tsW_0=\{\tW\in \tsW: W\in \sW_0\}$. $\tsW_0$ is also a closed set by \eqref{bound-prokhorov-metric}.
Also, 
\begin{eq}\label{eq:inf-tau-expr}
\phi_\tau (D,r) = \inf_{\tW\in \Gamma_{\geq r} \cap \tsW_0} J_{W_D}(\tW). 
\end{eq}
First, note that $J_{W_D} (\tW) = 0$ if and only if $\delcut(\tW,\tW_D) = 0$, which follows directly from Lemma~\ref{lem:rate-zero}. 
Thus, $\phi_\tau(D,r) = 0 $ for $r\in[0, l_{\tau}(D)]$. In this proof, let us henceforth assume $r\in (l_{\tau}(D),r_{\tau}(D)]$. It follows that $\Gamma_{\geq r} \cap \tsW_0 \neq \varnothing $ and $J_{W_D}$ is finite on $\Gamma_{\geq r} \cap \tsW_0 $. 
Consequently, $\phi_\tau(D,r)<\infty$.
For the strict positivity, since $\Gamma_{\geq r} \cap \tsW_0$ is compact and $J_{W_D}(\tW)$ is lower semi-continuous, the infimum in \eqref{eq:inf-tau-expr} is attained at some point $\tW^\star$. 
However, since $\tau (\tW^\star) \geq r> l_{\tau} (D)$, it must be that $\delcut (\tW_D,\tW^\star)>0$ and thus $J_{W_D}(\tW^\star)>0$. This shows that $\phi_\tau(D,r)$ is strictly positive. 

To prove the left-continuity of $\phi_{\tau}$, let $\alpha<\infty$ be such that $\phi_\tau(D,r')\leq \alpha$ for all $r'< r$.  
Recall that $F_{\star,r}\subset \Gamma_{\geq r} \cap \tsW_0$ is the set of minimizers of~\eqref{eq:inf-tau-expr}, which is shown to be non-empty above, and let   $\tW_r\in F_{\star,r}$.
Note that $J_{W_D}(\tW_{r'})\leq \alpha$, $\tau(\tW_{r'}) \geq r'$, and further, $\{\tilde{W}_{r'}:r'<r\}$ is precompact in $(\tsW,\delcut)$.
Take a subsequence along which as $r' \nearrow r$, $\tW_{r'} \to \tW$ in $(\sW,\delcut)$.
Then, by the lower semi-continuity of $J_{W_D}$, $J_{W_D}(\tW)\leq \alpha$, and by the continuity of $\tau$, $\tau(\tW) \geq r$. Thus $\phi_\tau(D,r) \leq \alpha$. This proves the left-continuity of $\phi_\tau(D,\cdot)$.

\end{proof}

%

\begin{proof}[Proof of Corollary~\ref{cor:ldp-cont-func}~(2)]
Let $\Gamma_{> r} = \{\tW : \tau (\tW) > r\}$. Then  Theorem~\ref{thm:ldp-given-degree} yields,
\begin{eq}
 - \lim_{r' \searrow r}\phi_\tau (D,r) &= - \inf_{\tW\in \Gamma_{>r}} J_D(\tW)\leq \liminf_{n\to\infty} \frac{1}{n^2} \log \PR(\tau_{n,\bd} >r) \\
 &\leq \limsup_{n\to\infty} \frac{1}{n^2} \log \PR(\tau_{n,\bd} \geq r) \leq - \phi_\tau (D,r). 
\end{eq}
Thus, if $r$ is a right-continuity point of $\phi_\tau (D,\cdot) $, then all the inequalities above hold with equality and the proof follows.
\end{proof}

\begin{proof}[Proof of Corollary~\ref{cor:ldp-cont-func}~(3)]
Let $\alpha = \phi_\tau (D,r)$. Recall that $\tbB(\tW,\varepsilon) $ denotes the $\varepsilon$ ball around~$\tW$ in $(\tsW, \delcut)$. Define $\Gamma_{r,\varepsilon} = \Gamma_{\geq r} \cap \big(\cap_{\tW \in F_{\star,r}} \tbB(\tW,\varepsilon)^c\big)$.
Note that
\begin{eq}
 \{\delcut(W^{\sss G_{n,\bd}},F_{\star,r}) \geq \varepsilon \text{ and }\tau_{n,\bd} \geq r\} = \{W^{\sss G_{n,\bd}}\in \Gamma_{r,\varepsilon}\}.
\end{eq}
It is enough to show that 
\begin{eq}
 \limsup_{n\to\infty}\frac{1}{n^2} \log \PR(W^{\sss G_{n,\bd}}\in \Gamma_{r,\varepsilon}) < -\alpha.
\end{eq}
Since $\Gamma_{r,\varepsilon}$ is a closed set, using Theorem~\ref{thm:ldp-given-degree}, it is enough to show that $\inf_{\tW \in \Gamma_{r,\varepsilon}} J_D(\tW) \leq \alpha$ yields a contradiction. 
Now, since $\Gamma_{r,\varepsilon}$ is compact and $J_D$ is lower semi-continuous, $J_D(\tW_r) \leq \alpha$ for some $\tW_r\in \Gamma_{r,\varepsilon}$.
Further,  
\begin{eq}
 F_{\star,r} = \Gamma_{\geq r} \cap \{\tW: J_D(\tW) \leq \alpha\},
\end{eq}so that $\tW_r \in F_{\star,r}$. Together with $\tW_r\in \Gamma_{r,\varepsilon}$, this yields a contradiction.
\end{proof}

\subsection{Convergence of the microcanonical partition function}
We now complete the proof of Corollary~\ref{cor:partition-function} in this section. We first need the following lemma: 
\begin{lemma}\label{lem:graph-count}
Recall that $\cG_{n,\bd}$ is the space of graphs with degree sequence $\bd$. Under {\rm Assumption~\ref{assmp:degree}}, as $n\to\infty$,  
\begin{eq}
\frac{1}{n^2} \log |\cG_{n,\bd}| \to h_{e}(W_D) = -\int_0^1 \beta(x)D(x) \dif x + \frac{1}{2}\int_{[0,1]^2} \log(1+\e^{\beta(x)+\beta(y)})\dif x \dif y,
\end{eq}where $h_e$ is defined in \eqref{defn:entropy-single}, and $\beta$ is given by {\rm Proposition~\ref{prop:graphon-limit}}.
\end{lemma}
\begin{proof}
Recall the definitions of $\hat{\bld{\beta}}$, $\hat{p}_{ij}$, $\hat{G}_n$, $W_{n,\bd}$, $D_n$ and $\beta_n$ from Section~\ref{sec:CDS-facts}.  
Note that  
\begin{eq}
\PR(\hat{G}_n = G) = \frac{\e^{\sum_{i\in [n]}\hat{\beta}_id_i }}{ \prod_{i<j} (1+ \e^{\hat{\beta}_i + \hat{\beta}_j})}, \qquad G\in \cG_{n,\bd}.
\end{eq}
Thus, if $\bld{d}(\hat{G}_n)$ denotes the degree sequnce of $\hat{G_n}$, then
\begin{eq} \label{eq:deg-partition}
\PR(\bld{d}(\hat{G}_n) = \bld{d}) = |\mathcal{G}_{n,\bd}| \frac{\e^{\sum_{i\in [n]}\hat{\beta}_id_i }}{ \prod_{i<j} (1+ \e^{\hat{\beta}_i + \hat{\beta}_j})}. 
\end{eq}
Now, using \eqref{eq:graphon-l1-conv},  $\beta_n \to \beta$ in $L_1$ and therefore
\begin{eq}
&\frac{1}{n^2} \log \prod_{i<j} (1+ \e^{\hat{\beta}_i + \hat{\beta}_j}) = 
\frac{1}{n^2}\sum_{i<j} \log (1+ \e^{\hat{\beta}_i + \hat{\beta}_j}) \\
& \quad = \frac{1}{2}\int_{[0,1]^2} \log(1+\e^{\beta_n(x)+\beta_n(y)})  \dif x \dif y - \frac{1}{n^2}\sum_{i\in [n]} \log (1+\e^{2\hat{\beta}_i}) \\
& \quad \to \frac{1}{2}\int_{[0,1]^2} \log(1+\e^{\beta(x)+\beta(y)})\dif x \dif y,
\end{eq}where the second term in the third equality goes to zero by dominated convergence theorem.
Moreover, using the fact that $D_n\to D$ in $L_1$ from Assumption~\ref{assmp:degree}, and that $d_i < n$, $\|\hat{\bld{\beta}}\|_{\infty} \leq C$,  it follows that
\begin{eq}
\frac{1}{n^2} \sum_{i\in [n]} \hat{\beta}_i d_i = \int_0^1 \beta_n(x)D_n(x) \dif x \to \int_0^1 \beta(x)D(x) \dif x.
\end{eq}
Now,
\begin{eq}\label{eq:entropy-simplification}
&h_e(W_D) \\
&= - \frac{1}{2} \int_{[0,1]^2} \big(W_D(x,y) \log(W_D(x,y)) + (1-W_{D}(x,y)) \log (1-W_{D}(x,y)) \big)\dif x \dif y \\
& = -\int_0^1 \beta(x)D(x) \dif x + \frac{1}{2}\int_{[0,1]^2} \log(1+\e^{\beta(x)+\beta(y)})\dif x \dif y.
\end{eq}
Now, turning back to \eqref{eq:deg-partition}, let us recall from Lemma~\ref{lem:CDS-counting} that $\PR(\bld{d}(\hat{G}_n) = \bld{d})$ lies in~$(e^{-n^{7/4}} ,1)$. Thus, 
\begin{eq}
\frac{1}{n^2} \log |\cG_{n,\bd}|  = -\frac{1}{n^2} \sum_{i\in [n]} \hat{\beta}_i d_i + \frac{1}{n^2} \log \prod_{i<j} (1+ \e^{\hat{\beta}_i + \hat{\beta}_j}) +o(1) \to h_e(W_D),
\end{eq}where the last step follows from \eqref{eq:entropy-simplification}.
The proof is now complete.

\end{proof}


\begin{proof}[Proof of Corollary~\ref{cor:partition-function}]  
We identify graphs with the corresponding empirical graphons--- this naturally embeds $\cG_{n,\bd}$ into the space $\tsW$.  The image of $\cG_{n,\bd}$ under this embedding map is henceforth denoted as $\tilde{\cG}_{n,\bd}$. 
For any $\tA \subseteq \tsW$, define $\tA_n = \tA \cap \tilde{\cG}_{n,\bd}$, so that $|\tA_n| <\infty$ for all $n$. 
Observe that 
\begin{eq}
\tPRD(\tA) = \frac{|\tA_n|}{|\cG_{n,\bd}|}.
\end{eq}
Therefore, using Theorem~\ref{thm:ldp-given-degree} together with Lemma~\ref{lem:graph-count},  for any  closed set $\tF \subset \tsW$ and open set $\tU \subset \tsW$,  
\begin{gather}
\limsup_{n\to\infty} \frac{1}{n^2} \log |\tF_n| \leq - \inf_{\tW\in \tF} J_D(\tW) + h_{e}(W_D) ,\label{ub-card-1}\\ 
\liminf_{n\to\infty} \frac{1}{n^2} \log |\tU_n| \geq - \inf_{\tW\in \tU} J_D(\tW) + h_{e}(W_D)\label{lb-card-2}.
\end{gather}
Fix $\varepsilon > 0$. Since $\tau$ is bounded, there exists $(a_i)_{i=1}^k$ such that the range of $\tau$ is a subset of $\cup_{i\in [k]}[a_i,a_i+\varepsilon]$. 
Now, let $\tF^{a_i}:= \tau^{-1} ([a_i,a_i+\varepsilon])$, which is closed due to the continuity of~$\tau$. Thus,
\begin{eq}
\e^{n^2 Z_{n,\tau}} \leq \sum_{i\in [k]} \e^{n^2 (a_i+\varepsilon)} |\tF^{a_i}| \leq k \max_{i\in [k]}  \e^{n^2 (a_i+\varepsilon)} |\tF^{a_i}|.
\end{eq}
Thus, \eqref{ub-card-1} implies that
\begin{eq}
\limsup_{n\to\infty} Z_{n,\tau} &\leq \max_{i\in [k]} \Big(a_i +\varepsilon - \inf_{\tW\in \tF^{a_i}} J_D(\tW) \Big) + h_{e}(W_D) \\
& \leq \varepsilon + \max_{i\in [k]} \sup_{\tW\in \tF^{a_i}} \big(\tau (\tW) - J_{D}(\tW)\big)+ h_{e}(W_D) \\
&= \varepsilon + \sup_{\tW\in \tsW} \big(\tau (\tW) - J_{D}(\tW)\big)+ h_{e}(W_D),
\end{eq}
where in the second step we have used the fact that $\tau(\tW) \geq a$ for all $\tW\in \tF^{a_i}$. 
For the lower bound, let $\tU^{b_i} = \tau^{-1} ((b_i,b_i+\varepsilon))$ for $i \leq l$ be such that $\cup_{i\in [l]}(b_i,b_i+\varepsilon)$ covers the range of $\tau$. An identical computation to above yields that
\begin{eq}
\liminf_{n\to\infty} Z_{n,\tau} \geq - \varepsilon + \sup_{\tW\in \tsW} \big(\tau (\tW) - J_{D}(\tW)\big)+ h_{e}(W_D).
\end{eq}
The proof of \eqref{eq:limit-partition-function} now follows by taking $\varepsilon \to 0$.
To see \eqref{eq:limit-number-graph}, the continuity of $\tau$, together with \eqref{ub-card-1} implies that
\begin{eq}
\limsup_{n\to\infty} \frac{1}{n^2} \log N_{n,\tau} (\bd,r) \leq - \phi_\tau (D,r) + h_e(W_D).
\end{eq}
Also, $N_{n,\tau}(\bd,r)$ is at least the number of graphs with degree sequence $\bd$ and $\tau (\tW) >r$. 
Thus, \eqref{lb-card-2} implies that
\begin{eq}
\liminf_{n\to\infty} \frac{1}{n^2} \log N_{n,\tau} (\bd,r) \geq  - \lim_{r' \searrow r}\phi_\tau (D,r) + h_e(W_D).
\end{eq}
The proof of \eqref{eq:limit-number-graph} is now complete using the right continuity of $\phi_{\tau} (D,\cdot)$ at $r$.
\end{proof}

\paragraph{Acknowledgements.} The authors gratefully thank an anonymous referee for an extremely thorough review, which has significantly improved the exposition of this paper. This work was initiated  during the BIRS workshop  ``Spin glasses and Related topics (2018)". 
The authors thank Amir  Dembo, Christian Borgs and Jennifer Chayes for motivating this research direction, and 
Julia Gaudio and Samantha Petti for pointing out minor errors in an earlier version of this paper.

\bibliographystyle{apa}
\bibliography{references}
\appendix 


\end{document}